\documentclass[12pt,reqno]{amsart}
\usepackage{fullpage}
\usepackage{times}
\usepackage{amssymb,amsmath}
\usepackage{color}
\usepackage{bbm}
\usepackage{dsfont}
\usepackage{multirow}
\usepackage[colorlinks=true, pdfstartview=FitH, linkcolor=blue, citecolor=blue, urlcolor=blue]{hyperref}
\usepackage{tikz-cd}

\newtheorem{theorem}{Theorem}[section]
\newtheorem{lemma}[theorem]{Lemma}
\newtheorem{prop}[theorem]{Proposition}
\newtheorem{cor}[theorem]{Corollary}
\theoremstyle{definition}
\newtheorem{remark}[theorem]{Remark}
\newtheorem{definition}[theorem]{Definition}
\newtheorem{notation}[theorem]{Notation}

\renewcommand{\mod}[1]{{\ifmmode\text{\rm\ (mod~$#1$)}\else\discretionary{}{}{\hbox{ }}\rm(mod~$#1$)\fi}}
\newcommand{\ep}{\varepsilon}

\newcommand{\spn}{\mathop{\rm span}}
\newcommand{\li}{\mathop{\rm li}}
\newcommand{\ord}{\mathop{\rm ord}}

\newcommand{\C}{{\mathbb C}}
\newcommand{\E}{{\mathbb E}}

\newcommand{\Q}{{\mathbb Q}}
\newcommand{\R}{{\mathbb R}}
\newcommand{\T}{{\mathbb T}}
\newcommand{\Z}{{\mathbb Z}}

\newcommand{\bA}{{\mathbf A}}

\newcommand{\bE}{{\mathbf E}}

\newcommand{\bM}{{\mathbf M}}

\newcommand{\bU}{{\mathbf U}}

\newcommand{\ba}{{\mathbf a}}
\newcommand{\bb}{{\mathbf b}}

\newcommand{\bn}{{\mathbf n}}

\newcommand{\bt}{{\mathbf t}}
\newcommand{\bu}{{\mathbf u}}
\newcommand{\bv}{{\mathbf v}}
\newcommand{\bw}{{\mathbf w}}
\newcommand{\bx}{{\mathbf x}}
\newcommand{\by}{{\mathbf y}}
\newcommand{\bz}{{\mathbf z}}
\newcommand{\bzero}{{\mathbf 0}}

\newcommand{\cA}{{\mathcal A}}
\newcommand{\cB}{{\mathcal B}}
\newcommand{\cC}{{\mathcal C}}

\newcommand{\cG}{{\mathcal G}}
\newcommand{\cK}{{\mathcal K}}

\newcommand{\cN}{{\mathcal N}}
\newcommand{\cP}{{\mathcal P}}
\newcommand{\cQ}{{\mathcal Q}}

\newcommand{\cS}{{\mathcal S}}

\newcommand{\cU}{{\mathcal U}}
\newcommand{\cV}{{\mathcal V}}
\newcommand{\cW}{{\mathcal W}}
\newcommand{\cX}{{\mathcal X}}
\newcommand{\cY}{{\mathcal Y}}
\newcommand{\cZ}{{\mathcal Z}}
\newcommand{\one}{{\mathds{1}}}

\begin{document}

\title{Inclusive prime number races}
\author[Greg Martin and Nathan Ng]{Greg Martin and Nathan Ng}
\address{Department of Mathematics \\ University of British Columbia \\ Room 121, 1984 Mathematics Road \\ Vancouver, BC \ V6T 1Z2 \\ Canada}
\email{gerg@math.ubc.ca}
\address{University of Lethbridge \\ Department of Mathematics and Computer Science \\ 4401 University Drive \\ Lethbridge, AB \ T1K 3M4 \\ Canada}
\email{nathan.ng@uleth.ca}
\date{\today}
\begin{abstract}
Let $\pi(x;q,a)$ denote the number of primes up to~$x$ that are congruent to~$a$ modulo~$q$.
A {\it prime number race}, for fixed modulus $q$ and residue classes $a_1, \ldots, a_r$, investigates
the system of inequalities $\pi(x;q,a_1) > \pi(x;q,a_2) > \cdots > \pi(x;q,a_r)$.
The study of prime number races was initiated by Chebyshev and further studied by many others, including Littlewood, Shanks--R\'{e}nyi, Knapowski--Turan,
and Kaczorowski. 
We expect that this system of inequalities should have arbitrarily large solutions $x$, and moreover we expect the same to be true no matter how we permute the residue classes~$a_j$; 
if this is the case, and if the logarithmic density of the set of such $x$ exists and is positive, the prime number race is called {\it inclusive}. 
In breakthrough research, 
Rubinstein and Sarnak~\cite{RS} proved conditionally that every prime number race is inclusive; they assumed not only the generalized Riemann hypothesis but also a strong statement about the linear independence of the zeros of Dirichlet L-functions. We show that the same conclusion can be reached assuming the generalized Riemann hypothesis and a substantially weaker linear independence hypothesis. In fact,  we can assume that almost all of the zeros may be involved in $\mathbb{Q}$-linear relations; and we can also conclude more strongly that the associated limiting distribution has mass everywhere.  This work makes use of a number of ideas from probability, the explicit formula from number theory, 
and the Kronecker--Weyl equidistribution theorem.
\end{abstract}

\maketitle

\section{Introduction}

A prime number race is the study of inequalities among the counting functions of primes in arithmetic progressions. Letting $\pi(x;q,a)$ denote as usual the number of primes up to $x$ that are congruent to $a\mod q$, we wish to understand for a given set $\{a_1,\dots,a_r\}$ how often (or indeed whether) the inequalities $\pi(x;q, a_1) > \pi(x;q, a_2) > \dots > \pi(x;q, a_r)$ are simultaneously satisfied. The survey article~\cite{GM} of Granville and the first author is a good starting reference for this subject in comparative prime number theory. We are interested in conditions under which we can establish that all permutations of the above string of inequalities occur with reasonable frequency; this goal motivates the following definitions.

\begin{definition} \label{inclusive def}
Let $a_1,\dots,a_r$ be distinct reduced residues\mod q.
\begin{enumerate}
\item We say that the prime number race among $a_1,\dots,a_r\mod q$ is {\em exhaustive} if, for every permutation $(\sigma_1, \dots, \sigma_r)$ of $(a_1,\dots,a_r)$, there are arbitrarily large real numbers $x$ for which
\begin{equation}
\pi(x;q, \sigma_1) > \dots > \pi(x;q, \sigma_r).
\label{chain of inequalities}
\end{equation}
\item We say that this prime number race is {\em weakly inclusive} if the logarithmic density of the set of real numbers $x$ satisfying the chain of inequalities~\eqref{chain of inequalities} exists for every permutation $(\sigma_1, \dots, \sigma_r)$ of $(a_1,\dots,a_r)$. (Recall that the logarithmic density $\delta(\cP)$ of a set $\cP$ of positive real numbers is
\[
\delta(\cP) :=
\lim_{x\to\infty} \bigg( \frac1{\log x} \int\limits_{\substack{1\le t\le x \\ t\in \cP}} \frac{dt}t \bigg) =
\lim_{y\to\infty} \bigg( \frac1{y} \int\limits_{\substack{0\le t\le y \\ e^t\in \cP}} dt \bigg)
\]
when the limit exists.)
\item We say that the prime number race is {\em inclusive} if these logarithmic densities exist and are all positive. Note that it is conceivable that a prime number race could be weakly inclusive yet not exhaustive (if one or more of the logarithmic densities equaled $0$); however, an inclusive prime number race is automatically exhaustive (and, of course, also weakly inclusive).
\item Finally, define the normalized error term
\begin{equation}
  \label{Exqa}
E(x;q,a) = \frac{\log x}{\sqrt x} \big( \phi(q) \pi(x;q,a) - \pi(x) \big)
\end{equation}
and the $\R^r$-valued function
\begin{equation}
  \label{Ey}
  E(x) = \big( E(x;q,a_1), \ldots, E(x;q,a_r) \big).
\end{equation}
We say that the prime number race among $a_1,\dots,a_r\mod q$ is {\em strongly inclusive} if, for every open ball $\cB_\rho(\bx) \subset \R^r$ of arbitrarily small radius~$\rho>0$ centered at an arbitrary point~$\bx\in\R^r$, the logarithmic density of the set of positive real numbers $x$ such that $E(x)\in\cB_\rho(\bx)$ exists and is positive.

Note that the sum of all $\phi(q)$ normalized error terms of the form~\eqref{Exqa} is exactly equal to $-\#\{p\mid q\}(\log x)/\sqrt x$ when $x>q$, which tends to~$0$ with~$x$. Consequently, when $r=\phi(q)$, we must modify this definition of strongly inclusive to refer only to open balls in $\R^{\phi(q)}$ that intersect the hyperplane $x_1+\cdots+x_{\phi(q)}=0$, or equivalently to open balls centered at points in this hyperplane. (We remark that if we replaced $\pi(x)$ with the logarithmic integral $\li(x)$ in equation~\eqref{Exqa}, then we would not need this stipulation in the case $r=\phi(q)$; however, using $\pi(x)$ is standard in this context.)

Note that the wedge
\begin{equation} \label{wedge definition}
\cS = \{(x_1,\dots,x_r)\in \R^r\colon x_1 > \cdots > x_r\}
\end{equation}
is an open subset of $\R^r$ (and intersects the relevant hyperplane), as are all similar wedges obtained by permuting coordinates; therefore any strongly inclusive race is also inclusive (barring some pathological situation where the logarithmic densities corresponding to every ball existed yet the logarithmic density corresponding to~$\cS$ did not exist), since $E(x)$ lying in one of these wedges is equivalent to the inequalities~\eqref{chain of inequalities}.
\end{enumerate}
In order to show that the prime number race among $a_1, \ldots, a_r \mod q$ is exhaustive, we must show that there are arbitrarily large values of $x$ for which $E(x)$ lies in the wedge~$\cS$.
(Indeed, we must show this for any permutation of $\{a_1,\dots,a_r\}$, but our arguments will never depend upon the identities of the $a_j$.) To show that this race is weakly inclusive, we must prove that the logarithmic density of the set $\{x\ge1\colon E(x)\in\cS\}$ exists; to show the race is inclusive, we must show that logarithmic density to be positive.
\end{definition}

There are relatively few known results establishing that a given race is exhaustive. 
In a tour de force, Littlewood established that the two-way races modulo~$3$ and modulo~$4$ are exhaustive.
Stark~\cite{St} showed for the first time that all two-way races modulo~$5$ are exhaustive.
Grosswald \cite[Theorem~5(iv)]{Gr}, building on work of Diamond~\cite{Di}, provided a condition on the zeros of Dirichlet $L$-functions modulo $q$ which implies that two-way races modulo $q$ are exhaustive.
Sneed~\cite{Sn}, using extensive calculation with the zeros of Dirichlet $L$-functions, showed that all two-way races modulo $q$ with $q \le 100$ are exhaustive.
K\'{a}tai~\cite{Kat} proved that if Dirichlet $L$-functions\mod q have no real zeros inside the critical strip, then any race between two squares or two nonsquares\mod q is exhaustive.
Under the same assumption, Knapowski and Tur\'an proved that all two-way races involving the residue class $1\mod q$ are exhaustive; Sneed~\cite{Sn} reports that Vorhauer proved the analogous statement for $-1\mod q$ (for $q$ sufficiently large) in unpublished work.
Kaczorowski~\cite[Corollary 5]{KaIV} establishes exhaustive two-way races under more complicated assumptions on the zeros of the $L(s,\chi)$.

While being exhaustive is probably the most natural property {\em a priori} to probe about prime number races, current methods in comparative prime number theory tend to establish the stronger properties in Definition~\ref{inclusive def}.
Indeed, in their seminal paper~\cite{RS}, Rubinstein and Sarnak proved, conditionally, that all prime number races are strongly inclusive; their results assumed not only the generalized Riemann hypothesis for Dirichlet $L$-functions (which we abbreviate as GRH), but also the following linear independence hypothesis, denoted LI:

\smallskip

\noindent {\bf Linear Independence Conjecture (LI)}. 
The non-negative ordinates (imaginary parts) of all zeros of Dirichlet $L$-functions with conductor $q$ are linearly independent over the rational numbers. 

\smallskip\noindent
This conjecture seems to have first been mentioned by Wintner in \cite{Wi2a}.  The conjecture became better known 
after Ingham~\cite{In2} used it to show that the summatory function of the M\"{o}bius function is not bounded by $C \sqrt{x}$ for any positive~$C$.  
For a more comprehensive history of LI, see the introduction in \cite{MaNg}. 

In this paper we also assume the generalized Riemann hypothesis (GRH) throughout, for all Dirichlet $L$-functions corresponding to characters (primitive and imprimitive) modulo~$q$, although most of our results do not require RH for $\zeta(s)$ itself. Our aim is to substantially weaken the linear independence hypothesis. In order to state our results, we need some notation for the ordinates of these zeros of $L$-functions, as well as some terminology to describe when a particular ordinate of a zero of a Dirichlet $L$-function is involved in a linear relation with other ordinates of zeros.

\begin{notation}
We use the following notation for multisets of ordinates of zeros of Dirichlet $L$-functions:
\[
\Gamma(\chi) = \{ \gamma >0: L(\tfrac12+i\gamma,\chi)=0\} \qquad\text{and}\qquad \Gamma(q) = \bigcup_{\substack{\chi\mod q \\ \chi\ne\chi_0}} \Gamma(\chi).
\]
(Since we will be assuming GRH, restricting the real part of the zeros to $\frac12$ is natural.)
Note that if $L(\frac12-i\gamma,\chi)=0$, then $L(\frac12+i\gamma,\overline\chi)=0$ by the functional equation; for this reason, we will only need to consider positive ordinates $\gamma$ in these multisets. It is conceivable that some $L(\frac12,\chi)$ could vanish, but we do not include $0$ in these multisets---any hypothetical zeros at $s=\frac12$ will be dealt with explicitly in the formulas to come.
\end{notation}

\begin{definition} \label{self-sufficient def}
We say that $\gamma\in\Gamma(q)$
is {\em self-sufficient} if $\gamma$ cannot be written as a finite $\Q$-linear combination of elements of $\Gamma(q)
\setminus\{\gamma\}$. For example, if there exist characters $\chi_1,\chi_2\mod q$, not necessarily distinct, such that $L(\frac12+i\gamma,\chi_1)=L(\frac12+2i\gamma,\chi_2)=0$, then $\gamma\in\Gamma(q)$ is automatically not self-sufficient, and the same for~$2\gamma$. If $\gamma\in\Gamma(q)$ is self-sufficient, then in particular $\tfrac12+i\gamma$ is a simple zero of $\prod_{\chi\mod q,\, \chi\ne\chi_0} L(s,\chi)$.
\end{definition}

\begin{notation}
We introduce the notation
\begin{align*}
\Gamma^S(\chi) & = \{ \gamma\in\Gamma(\chi) \colon \gamma \text{ is self-sufficient} \}, \\
\Gamma^S(q) & = \bigcup_{\substack{\chi\mod q \\ \chi\ne\chi_0}} \Gamma^S(\chi).
\end{align*}
Notice for example that the statement ``every $\gamma\in\Gamma^S(q)$ is self-sufficient'' is stronger than the statement that $\Gamma^S(q)$ is a linearly independent set over~$\Q$, since the former statement also considers linear combinations involving elements of $\Gamma(q) \setminus \Gamma^S(q)$ (see Remark~\ref{stronger remark} below for a related observation). In fact, $\Gamma^S(q)$ is the intersection of all maximal linearly independent subsets of $\Gamma(q)$; equivalently~\cite[Theorem~1.12]{FIS}, $\Gamma^S(q)$ is the intersection of all subsets of $\Gamma(q)$ that are bases for the $\Q$-vector space generated by $\Gamma(q)$. In the language of matroid theory, $\Gamma(q)$ represents a finitary matroid and $\Gamma^S(q)$ is precisely the set of coloops of that matroid.
\end{notation}

We begin by stating two tidy results; both of these follow from our most general result, the statement of which (Theorem~\ref{inclusive theorem}) we delay momentarily for the purposes of exposition. The simplest prime number race is between a pair of contestants (that is, $r=2$); we note for example that a two-way race is exhaustive precisely when the difference $\pi(x;q,a)-\pi(x;q,b)$ changes sign infinitely often. Our first theorem gives conditions under which we can deduce that two-way races are weakly inclusive, inclusive, or strongly inclusive:

\begin{theorem}
\label{two-way theorem}
Assume GRH. Let $a$ and $b$ be distinct reduced residues\mod q.
\begin{enumerate}
\item If the set
$$\bigcup_{\substack{\chi\mod q \\ \chi(a)\ne\chi(b)}} \Gamma^S(\chi)$$
has at least three elements, then the prime number race between $a$ and $b\mod q$ is weakly inclusive.
\item There exists a constant $W(q)$ such that if
\begin{equation}
\sum_{\substack{\chi\mod q \\ \chi(a)\ne\chi(b)}} \sum_{\gamma\in\Gamma^S(\chi)} \frac1\gamma > W(q),
\label{somebody's W-robust}
\end{equation}
then the prime number race between $a$ and $b\mod q$ is inclusive.
\item If the sum
\begin{equation}
\sum_{\substack{\chi\mod q \\ \chi(a)\ne\pm\chi(b)}} \sum_{\gamma\in\Gamma^S(\chi)} \frac1\gamma
\label{somebody's robust}
\end{equation}
diverges, then the prime number race between $a$ and $b\mod q$ is strongly inclusive.
\end{enumerate}
\end{theorem}

The total number of zeros of Dirichlet $L$-functions\mod q in the critical strip whose imaginary parts are between 0 and $T$ is asymptotic to a constant times $T\log T$~\cite[Corollary 14.7]{MV}. Theorem~\ref{two-way theorem} reveals that even if the set of self-sufficient ordinates is so thin that there are only $\ep T(\log T)^{-1}$ of them up to height $T$, we can still conclude that a two-way prime number race is inclusive. Note that we do require, in parts~(a) and~(b) of the theorem, the condition that the ordinates are associated with characters satisfying $\chi(a)\ne\chi(b)$; the reader might find this restriction intuitive after recalling the formula
\begin{equation}
\pi(x;q,a) - \pi(x;q,b) = \frac1{\phi(q)} \sum_{\chi\mod q} \big( \overline{\chi}(a)-\overline{\chi}(b) \big) \pi(x,\chi),
\label{which chis matter}
\end{equation}
in which the zeros of those $L(s,\chi)$ for which $\chi(a)=\chi(b)$ never appear when the explicit formula for $\pi(x,\chi)$ is inserted. Part~(c) of the theorem will be deduced from Theorem~\ref{inclusive theorem}(c) below, at which point the slightly different condition $\chi(a)\ne\pm\chi(b)$ wil be explained (the more general sufficient condition is more complicated in this case).

\begin{remark}
The methods of Ingham~\cite{In2}, as elaborated on by Kaczorowski~\cite{KaIV}, can be shown to imply that the set of $x$ for which $\pi(x;q,a)>\pi(x;q,b)$ has positive lower logarithmic density, assuming GRH and a weaker linear independence condition than our part (b)---namely that there exists a linearly independent set $\cY$ of (not necessarily self-sufficient) positive ordinates of zeros of Dirichlet $L$-functions\mod q, among those for which $\chi(a)\ne\chi(b)$, for which $\sum_{\gamma\in \cY} 1/\gamma$ is sufficiently large in terms of~$q$. However, that method does not address the property of being strongly inclusive (nor can it establish that the logarithmic densities exist).
\end{remark}

We can formulate races involving any $r$ functions, not just functions of the form $\pi(x;q,a)$. Restricting to $r=2$, for example, we say that the race between two functions $f(x)$ and $g(x)$ is exhaustive if the function $f(x)-g(x)$ has arbitrarily large sign changes, is weakly inclusive if the logarithmic densities of the sets $\{x>0\colon f(x)>g(x)\}$ and $\{x>0\colon g(x)>f(x)\}$ exist, is inclusive if those logarithmic densities are positive, and is strongly inclusive if a suitably normalized version of $f(x)-g(x)$ lies in any subinterval of~$\R$ with positive logarithmic density. Our methods apply equally well to other two-contestant prime number races, as the following theorem indicates. Let $\li(x) = \int_2^\infty {dt}/{\log t}$ denote the usual logarithmic integral, and note that $\Gamma^S(\chi_0)$ denotes the set of self-sufficient ordinates of zeros of the Riemann zeta function (which is the Dirichlet $L$-function associated to the principal character~$\chi_0$ for $q=1$).

\begin{theorem}
\label{pili theorem}
Assume RH.
\begin{enumerate}
\item If the set $\Gamma^S(\chi_0)$ has at least three elements, then the race between $\pi(x)$ and $\li(x)$ is weakly inclusive.
\item There exists a constant $W(1)$ such that if
$
\sum_{\gamma\in\Gamma^S(\chi_0)} \frac1\gamma > W(1)
$, then the race between $\pi(x)$ and $\li(x)$ is inclusive.
\item If the sum
$
\sum_{\gamma\in\Gamma^S(\chi_0)} \frac1\gamma
$
diverges, then the race between $\pi(x)$ and $\li(x)$ is strongly inclusive.
\end{enumerate}
\end{theorem}

\noindent Part (a) says, in other words, that if $\zeta(s)$ has at least three self-sufficient zeros, then the logarithmic density of the sets $\{x>0\colon\pi(x)>\li(x)\}$ and $\{x>0\colon\pi(x)<\li(x)\}$ exist. Our methods also establish variants of Theorem~\ref{pili theorem} for primes in arithmetic progressions (see Theorem~\ref{pili theorem for APs}).

We move now to prime number races with more than two contestants.
To state our next result, we need to introduce some additional terminology.

\begin{definition} \label{robustdefn}
For any positive integer~$k$, we say that $\chi$ is {\em $k$-sturdy} if $\#\Gamma^S(\chi) \ge k$.
For any nonnegative real number~$W$, we say that $\chi$ is {\em $W$-robust} if $\sum_{\gamma\in \Gamma^S(\chi)} 1/\gamma \ge W$.
We say that $\chi$ is {\em robust} if $\sum_{\gamma\in \Gamma^S(\chi)} 1/\gamma$ diverges.
\end{definition}

\begin{remark}
We reiterate that the definition of self-sufficient, and hence the definitions of $k$-sturdy and robust, depend upon the chosen conductor~$q$. If $\chi^*$ is a character\mod q and $\chi$ is the character\mod{q^2} induced by $\chi^*$, for example, then $L(s,\chi^*)$ and $L(s,\chi)$ are exactly the same function with exactly the same zeros; and yet a particular ordinate $\gamma\in\Gamma(\chi^*) = \Gamma(\chi)$ might be self-sufficient\mod q but not self-sufficient\mod{q^2} (for example, if some primitive character\mod{q^2} also had $\gamma$ as the ordinate of one of its zeros). In our setting, however, the conductor $q$ will always remain fixed.
\end{remark}

\begin{theorem} \label{all-way theorem}
Assume GRH. Let $q\ge3$ be an integer.
\begin{enumerate}
\item If every nonprincipal character $\chi\mod q$ is $(2\phi(q)+1)$-sturdy, then every prime number race\mod q, including the full $\phi(q)$-way race, is weakly inclusive.
\item There exists a constant $W(q)$ such that if every nonprincipal character $\chi\mod q$ is $W(q)$-robust, then every prime number race\mod q, including the full $\phi(q)$-way race, is inclusive.
\item If every nonprincipal character $\chi\mod q$ is robust, then every prime number race\mod q, including the full $\phi(q)$-way race, is strongly inclusive.
\end{enumerate}
\end{theorem}

\noindent The irrelevance of the principal character $\chi_0$ is again intuitive upon examination of equation~\eqref{which chis matter}, in which the summand $\chi=\chi_0$ always vanishes.

Both Theorem~\ref{two-way theorem}(b)--(c) and Theorem~\ref{all-way theorem} are special cases of the following result, the proof of which is the main focus of the rest of this paper, and its corollary. (This following result implies a slightly weaker version of Theorem~\ref{two-way theorem}(a); we discuss our slightly stronger version, as well as Theorem~\ref{pili theorem} and its variants for arithmetic progressions, at the end of Section~\ref{log densities section}.) Let $\Re z$ and $\Im z$ denote the real and imaginary parts, respectively, of the complex number~$z$.

\begin{theorem}
\label{inclusive theorem}
Assume GRH. Let $r\ge2$ be an integer, and let $a_1,\dots,a_r$ be distinct reduced residues\mod q.
\begin{enumerate}
\item Suppose that the set of vectors
\begin{multline} \label{Re and Im vectors}
\big\{ (1,\dots,1) \big\} \cup \big\{ \big( \Re\chi(a_1), \dots, \Re\chi(a_r) \big) \colon \chi\mod q \text{ is $(2r+1)$-sturdy} \big\} \\
\cup \big\{ \big( \Im\chi(a_1), \dots, \Im\chi(a_r) \big) \colon \chi\mod q \text{ is $(2r+1)$-sturdy} \big\}
\end{multline}
spans the vector space $\R^r$. Then the prime number race among $a_1,\dots,a_r$\mod q is weakly inclusive.
\item There exists a constant $W(q)$ such that, if the set of vectors
\begin{multline*}
\big\{ (1,\dots,1) \big\} \cup \big\{ \big( \Re\chi(a_1), \dots, \Re\chi(a_r) \big) \colon \chi\mod q \text{ is $W(q)$-robust} \big\} \\
\cup \big\{ \big( \Im\chi(a_1), \dots, \Im\chi(a_r) \big) \colon \chi\mod q \text{ is $W(q)$-robust} \big\}
\end{multline*}
spans $\R^r$, then the prime number race among $a_1,\dots,a_r$\mod q is inclusive.
\item Suppose that the set of vectors
\begin{multline*}
\big\{ \big( \Re\chi(a_1), \dots, \Re\chi(a_r) \big) \colon \chi\mod q \text{ is robust},\, \chi\ne\chi_0 \big\} \\
\cup \big\{ \big( \Im\chi(a_1), \dots, \Im\chi(a_r) \big) \colon \chi\mod q \text{ is robust},\, \chi\ne\chi_0 \big\}
\end{multline*}
spans $\R^r$ when $r<\phi(q)$, or spans $\{x_1+\cdots+x_{\phi(q)}=0\}\subset\R^{\phi(q)}$ when $r=\phi(q)$. Then the prime number race among $a_1,\dots,a_r$\mod q is strongly inclusive.
\end{enumerate}
\end{theorem}

\noindent We point out that Theorem~\ref{two-way theorem}(c) follows from Theorem~\ref{inclusive theorem}(c) by the following argument: if the sum~\eqref{somebody's robust} diverges, then there exists at least one character $\chi\mod q$ with $\chi(a)\ne\pm\chi(b)$ that is robust. But note that $\chi(a)\ne\pm\chi(b)$ is equivalent to $\Im\big( \chi(a)\overline\chi(b) \big) \ne 0$ since the only real character values on reduced residues are $\pm1$; and a short calculation verifies that $\Im\big( \chi(a)\overline\chi(b) \big) \ne 0$ is equivalent to the linear independence of $(\Re\chi(a),\Re\chi(b))$ and $(\Im\chi(a),\Im\chi(b))$ over~$\R$.


A slightly simpler statement follows immediately from Theorem~\ref{inclusive theorem}:

\begin{cor}
\label{inclusive cor}
Assume GRH. Let $r\ge2$ be an integer, and let $a_1,\dots,a_r$ be distinct reduced residues\mod q.
\begin{enumerate}
\item Suppose that the set of vectors
\begin{equation*}
\big\{ (1,\dots,1) \big\} \cup \big\{ \big( \chi(a_1), \dots, \chi(a_r) \big) \colon \chi\mod q \text{ is $(2r+1)$-sturdy} \big\}
\end{equation*}
spans the vector space $\C^r$. Then the prime number race among $a_1,\dots,a_r$\mod q is weakly inclusive.
\item There exists a constant $W(q)$ such that, if the set of vectors
\begin{equation*}
\big\{ (1,\dots,1) \big\} \cup \big\{ \big( \chi(a_1), \dots, \chi(a_r) \big) \colon \chi\mod q \text{ is $W(q)$-robust} \big\}
\end{equation*}
spans $\C^r$, then the prime number race among $a_1,\dots,a_r$\mod q is inclusive.
\item Suppose that the set of vectors
\begin{equation*}
\big\{ \big( \chi(a_1), \dots, \chi(a_r) \big) \colon \chi\mod q \text{ is robust},\, \chi\ne\chi_0 \big\}
\end{equation*}
spans $\C^r$ when $r<\phi(q)$, or spans $\{z_1+\cdots+z_{\phi(q)}=0\} \subset \C^{\phi(q)}$ when $r=\phi(q)$. Then the prime number race among $a_1,\dots,a_r$\mod q is strongly inclusive.
\end{enumerate}
\end{cor}

Theorem~\ref{all-way theorem} follows from Corollary~\ref{inclusive cor} because the set $\big\{ \big( \chi(a_1), \dots, \chi(a_r) \big) \colon \chi\mod q \big\}$ always spans $\C^r$, by the orthogonality of Dirichlet characters (it is of course important here that $a_1,\dots,a_r$ are distinct modulo~$q$). Similarly, Theorem~\ref{two-way theorem}(b) follows from Corollary~\ref{inclusive cor}(b) because the divergence of the sum~\eqref{somebody's robust} is equivalent to the assertion that at least one character $\chi\mod q$ with $\chi(a)\ne\chi(b)$ is robust, whereupon the set $\big\{ (1,1), \big( \chi(a),\chi(b) \big) \big\}$ spans~$\C^2$.

\begin{remark}
Corollary~\ref{inclusive cor} is simpler to apply in practice, but Theorem~\ref{inclusive theorem} is indeed somewhat stronger. For example, take any three-way race modulo $5$ (say among $a_1,a_2,a_3$), and suppose that exactly one Dirichlet $L$-function\mod5 is robust, one that corresponds to a complex character $\chi\mod5$. Then one can check that
\[
\{ (1,1,1), (\Re\chi(a_1),\Re\chi(a_2),\Re\chi(a_3)), (\Im\chi(a_1),\Im\chi(a_2),\Im\chi(a_3))\}
\]
spans $\R^3$, and so Theorem~\ref{inclusive theorem}(b) tells us that this race is strongly inclusive; but we cannot reach that conclusion from Corollary~\ref{inclusive cor}, because $\{(1,1,1), (\chi(a_1),\chi(a_2),\chi(a_3)\}$ does not span $\C^3$.
\end{remark}

\begin{remark}
The method that we use to prove that prime number races are weakly inclusive actually yields, in every case, an additional conclusion as well: ``ties have density zero''. More precisely, every time we establish that a prime number race is weakly inclusive, we also establish the fact that every set of the form $\{x>0\colon \pi(x;q, a_j) = \pi(x;q,a_k)\}$ has logarithmic density zero. Equivalently, the $r!$ logarithmic densities of the sets satisfying the inequalities~\eqref{chain of inequalities}, corresponding to the $r!$ possible permutations, not only exist but sum to~$1$. Indeed, for any function $f(x)$ satisfying $f(x)=o(\sqrt x/\log x)$, our proofs that races are weakly inclusive actually show that the sets $\{x>0\colon |\pi(x;q, a_j) - \pi(x;q,a_k)| < f(x)\}$ have logarithmic density~$0$. Analogous comments apply to Theorem~\ref{pili theorem} and its variants for arithmetic progressions. Of course, races that are proved to be inclusive or strongly inclusive are certainly weakly inclusive as well and hence also have the property that ties have density zero.
\end{remark}

\begin{remark}  \label{variants of pi remark}
Up to this point, we have been using exclusively the function $\pi(x;q,a)$ that counts primes each with weight~$1$. Common weighted variants of this function are $\theta(x;q,a)$, which counts each relevant prime $p$ with weight $\log p$, and $\psi(x;q,a)$, which counts prime powers as well as primes via the von Mangoldt function $\Lambda(n)$. We remark that all of the theorems we prove herein also hold for prime number races using these weighted counting functions, that is, for inequalities of the form
\[
\theta(x;q, \sigma_1) > \dots > \theta(x;q, \sigma_r) \quad\text{and}\quad \psi(x;q, \sigma_1) > \dots > \psi(x;q, \sigma_r).
\]
By this we mean that if we replace every occurrence of $\pi$ with $\theta$ (or $\psi$) and every occurrence of $\li(x)$ with $x$, then all theorems in this paper remain valid. We comment on these variants at the end of Section~\ref{explicit formulae section}.
\end{remark}

In this paper, we are retaining GRH as a hypothesis but substantially weakening the linear independence hypothesis LI. One might also speculate whether it is possible to weaken or remove the assumption of GRH itself.
However, Ford and Konyagin~\cite{FK1,FK2} have shown that given a prime number race with at least three contestants, there exist specific points in the critical strip (with $\frac12 < \sigma < 1$) such that if Dirichlet $L$-functions\mod q have zeros at precisely those points, then the prime number race is not exhaustive (much less inclusive). Further joint work with Lamzouri~\cite{FKL} leads to a similar construction that would cause two-way prime number races to be not inclusive.
We have considered the problem of constructing analogous hypothetical configurations of zeros, satisfying GRH, that would force a prime number race to be not inclusive; the results of this paper show, however, that such configurations must necessarily be extremely complicated, in that 100\% of the zeros (of some Dirichlet $L$-functions modulo $q$, at least) would need be involved in linear combinations with one another.

To summarize our main goals: we need to establish all three parts of Theorem~\ref{inclusive theorem} and all three parts of Theorem~\ref{pili theorem}, as well as part~(a) of Theorem~\ref{two-way theorem}. (We have already deduced Theorem~\ref{two-way theorem}(c) and Corollary~\ref{inclusive cor} from Theorem~\ref{inclusive theorem}, and Theorem~\ref{two-way theorem}(b) and Theorem~\ref{all-way theorem} in turn follow from Corollary~\ref{inclusive cor}.) The~(a) parts of these results are established in Section~\ref{log densities section}, while their~(b) and~(c) parts are established in Section~\ref{cylinder section}. Furthermore, an analogue of Theorem~\ref{pili theorem} for the counting function of primes in a single arithmetic progression is given in Theorem~\ref{pili theorem for APs} below. In the next section we describe the notational conventions used throughout this paper, then give a mathematical outline of the structure of the proofs of our theorems; we conclude the next section by describing the contents of each remaining section.

\section{Notation, conventions, and structure of the proof}  \label{notation section}

Throughout this paper, we will fix a modulus $q\ge3$ and an integer $r$ in the range $2\le r\le\phi(q)$, and we will also fix integers $a_1,\dots,a_r$, all relatively prime to $q$, that represent $r$ distinct residue classes\mod q. Also throughout this paper, we will assume the generalized Riemann hypothesis for Dirichlet $L$-functions with conductor~$q$ (corresponding to both primitive and imprimitive characters, and thus including the Riemann zeta function, for example).

For functions $f(x)$ and $g(x)$ we interchangeably use the notations $f(x)=O(g(x))$ and $f(x) \ll g(x)$ and $g(x) \gg f(x)$ with their usual meanings, namely that there exist positive constants $x_0$ and $M$ such that $|f(x)| \le Mg(x)$ for all $x \ge x_0$. Since $q$ is fixed, the implicit constants in such expressions may depend on~$q$.

In mathematical expressions, we will use lowercase boldface letters such as $\bx$ to denote vectors, and uppercase boldface letters such as $\bM$ to denote matrices; we will also use uppercase calligraphic letters such as $\cB$ to denote sets.
The $n$-dimensional torus will be denoted $\mathbb{T}^n$ and the infinite dimensional torus 
will be denoted $\mathbb{T}^{\infty}$.

\begin{definition} \label{log dist def}
We say that a function $h\colon [1,\infty)\to\R^r$ possesses a {\em limiting logarithmic distribution} if there exists a probability measure $\mu$ on $\R^r$ such that 
\begin{equation}
  \lim_{x \to \infty} \bigg( \frac1{\log x} \int_1^x f(h(t)) \frac{dt}t \bigg) = \int_{\R^r} f(\bx) \,d \mu(\bx)
\end{equation}
for all bounded, continuous functions on~$\R^r$. A simple change of variables shows that the limiting logarithmic distribution~$\mu$ of~$h(t)$, when it exists, is the same as the usual limiting distribution of~$h(e^t)$:
\begin{equation}
  \lim_{y \to \infty} \bigg( \frac1y \int_0^y f(h(e^t)) \,dt \bigg) = \int_{\R^r} f(\bx) \,d \mu(\bx).
\end{equation}
\end{definition}

\begin{notation} \label{char fn and conv}
Given a probability measure $\mu$ on $\R^r$, we define its {\em characteristic function} (or Fourier transform) using the normalization
\begin{equation}
  \label{charfcn}
  \widehat\mu(\bt) = \int_{\R^r}  e^{i \bt \cdot \bx} \,d\mu(\bx)
\end{equation}
for $\bt  \in \R^r$. We also define the convolution $\mu*\nu$ of two probability measures $\mu$ and $\nu$ on $\R^r$ to be the measure on $\R^r$ satisfying
\begin{equation}  \label{convolution}
  (\mu*\nu)(\cZ) = \int_{\R^r} \int_{\R^r} \one_{\cZ}(\bx_1+\bx_2)  \,d\mu(\bx_1) \,d\nu(\bx_2)
\end{equation}
for any Borel subset $\cZ$ of~$\R^r$. Note that if if $\cX$ and $\cY$ are Borel subsets of $\R^r$ such that the pointwise sum $\cX+\cY$ is contained in~$\cZ$, then
\begin{equation}  \label{convolution lower bound}
  (\mu*\nu)(\cZ) \ge \int_{\R^r} \int_{\R^r} \one_{\cX}(\bx_1) \one_{\cY}(\bx_2)  \,d\mu(\bx_1) \,d\nu(\bx_2) = \mu(\cX) \nu(\cY),
\end{equation}
an inequality we will employ often.
\end{notation}

In the above definitions (and also in Appendix~\ref{probability appendix}), $\mu$ denotes a generic probability measure. However, for the rest of the paper, we will use $\mu$ only to denote a specific limiting distribution, defined momentarily in equation~\eqref{mu defined here}, which depends upon $q$ and $a_1,\dots,a_r$.

\subsection{Outline of the proofs} \label{outline section}

We now describe our approach to establishing Theorem~\ref{inclusive theorem} (from which Theorems~\ref{two-way theorem}(b), Theorem~\ref{all-way theorem}, and Corollary~\ref{inclusive cor} all follow), as well as the related Theorems~\ref{two-way theorem}(a) and~\ref{pili theorem} and the variant of the latter for arithmetic progressions (Theorem~\ref{pili theorem for APs}).

In their seminal paper~\cite{RS}, Rubinstein and Sarnak deduced from GRH that $E(x)$ (see equation~\eqref{Ey}) possesses a limiting logarithmic distribution $\mu$ on $\R^r$, so that
\begin{equation}  \label{mu defined here}
  \lim_{y \to \infty} \bigg( \frac1y \int_0^y f(E(e^t)) \,dt \bigg) = \int_{\R^r} f(\bx) \,d \mu(\bx)
\end{equation}
for any bounded continuous function~$f$. However, this by itself is not enough to show that the logarithmic density of the set $\{t\ge1\colon E(t)\in\cS\}$ exists: we would like to take $f(\bx)$ to be the indicator function $\one_\cS$ of the wedge $\cS$, since
\begin{equation}  \label{relevant delta}
\delta\big( \{t\ge1\colon E(t)\in\cS\} \big) = \lim_{y \to \infty} \bigg( \frac1y \int_0^y \one_\cS(E(e^t)) \,dt \bigg);
\end{equation}
but we cannot immediately do so since $\one_\cS$ is not continuous. (Indeed, this is the significant mathematical issue that must be addressed, to rule out the possibility that the distribution concentrates its mass on the boundary of the set~$\cS$.) By assuming LI, Rubinstein and Sarnak showed that $\mu$ is absolutely continuous with respect to Lebesgue measure on $\R^r$, and thus possesses a density function $g(\bx)$, meaning that $d\mu(\bx) = g(\bx)\,d\bx$. (We are lying slightly here and in the next 
few paragraphs for the purposes of exposition; at the end of this section we will own up to the lie.) This is already enough to show that the prime number race is weakly inclusive. Finally, from the decay of the characteristic function $\widehat\mu$, they conclude that the density function $g(\bx)$ is actually the restriction to real arguments of a function that is entire in each variable, which implies that the prime number race is strongly inclusive, essentially because entire functions cannot vanish on sets of positive measure.

The main innovation of this article is to prove similar results under much weaker assumptions on the linear independence of the ordinates of zeros of Dirichlet $L$-functions.

We begin by proving that the limiting logarithmic distribution $\mu$ described above can be written as a convolution $ \mu= \mu^R_W*\mu^N_W$ of two probability measures on $\R^r$
(see Proposition \ref{muconvolutionrobust}). Roughly speaking, $\mu^R_W$ corresponds to the self-sufficient zeros of the $W$-robust characters\mod q and $\mu^N_W$ to the rest of the zeros. Similarly, we can write $\mu=\mu_k^S*\mu_k^N$ where $\mu_k^S$ corresponds to the self-sufficient zeros of the $k$-sturdy characters
(see Proposition \ref{muconvolutionsturdy}). We accomplish this via a close look at the proof of the Kronecker--Weyl theorem, which says that a line of the form $\{ y(\xi_1, \ldots, \xi_m) \colon y \in \R \}$ inside the $m$-dimensional torus $\T^m$ is equidistributed in some subtorus~$\cA$ determined by the rational linear relations among the~$\xi_j$. We prove that if the set $\{\xi_1, \ldots, \xi_m\}$ can be partitioned into two subsets that are relatively independent, in the sense of Definition~\ref{relind} below, then the limiting subtorus decomposes as the direct sum of two smaller subtori; this decomposition allows our limiting logarithmic distributions to be written as convolutions.

Next, we show that $\mu_k^S$ is absolutely continuous with respect to Lebesgue measure on~$\R^r$. 
We do so in a similar way to Rubinstein and Sarnak, namely by showing that the characteristic function $\widehat\mu_k^S$ decays rapidly enough to be absolutely integrable. We then show that the convolution $\mu_k^S*\nu$ remains absolutely continuous for any probability measure $\nu$. In particular, $\mu=\mu_k^S*\mu_k^N$ is absolutely continuous with respect to Lebesgue measure on~$\R^r$, and we again conclude that the logarithmic density of the set $\{t\ge1\colon E(t)\in\cS\}$ exists, where $\cS$ is the wedge defined in equation~\eqref{wedge definition}.

Finally, we establish that the support of $\mu^R_W$ becomes, as~$W$ increases, large enough to include any prescribed point in~$\R^r$. It follows that the same is true of any convolution $\mu^R_W*\nu$; thus in particular, $\mu=\mu^R_W*\mu^N_W$ must assign positive mass to the wedge $\cS$ (regardless of any biases stemming from quadratic residues vs.~nonresidues, hypothetical zeros at $s=\frac12$, or the unknown properties of~$\mu^N_W$), thereby showing that the prime number race is inclusive. There is an analogous density~$\mu^R$ corresponding to the self-sufficient zeros of robust characters (the case ``$W=\infty$'') which is relevant to showing that prime number races are strongly inclusive (although, in practice, we derive those results from $\mu^R_W$ itself for finite but arbitrarily large values of~$W$).
It is possible to use the same method as Rubinstein and Sarnak to show that~$\mu^R$ is supported on all of $\R^r$, namely by showing that $\widehat\mu^R$ decays sufficiently rapidly. However, doing so would require a more stringent definition of robustness (roughly speaking, we would need about~$T$ of the zeros up to height $T$ to be self-sufficient, as opposed to the $T/\log T$ or so implicit in our actual definition). Instead, we use the characteristic function $\widehat\mu^R_W$ to write down a concrete $\R^r$-valued random variable whose distribution is also $\mu^R_W$, and then we show directly that the support of this random variable is sufficiently large. Roughly speaking, the condition that the vectors in equation~\eqref{Re and Im vectors} span~$\R^r$ yields that the random variable is supported ``in all directions'', while the $W$-robustness of each relevant character implies that the random variable reaches sufficiently far ``in that character's directions''.

We now reveal and rectify the slight lies in the above exposition.
When discussing Rubinstein and Sarnak's result, we described the limiting logarithmic distribution of $E(x)$ as being absolutely continuous with respect to Lebesgue measure on~$\R^r$. This is true as long as $r<\phi(q)$; but if $r=\phi(q)$, so that we are racing all the reduced residue classes\mod q against one another, then it is false because $\sum_{(a,q)=1} E(x;q,a) = o(1)$ as discussed in Definition~\ref{inclusive def}(d). It follows that the limiting logarithmic distribution $\mu$ must be supported on the hyperplane $\cW=\{x_1+\cdots+x_r=0\}$, hence cannot be absolutely continuous with respect to Lebesgue measure on~$\R^r$. However, it does turn out to be absolutely continuous with respect to Lebesgue measure on $\cW$, which is sufficient since $\cW$ intersects the wedge $\cS$ and all similar wedges produced by permuting coordinates. Indeed, these wedges are translation-invariant in the direction $(1,\dots,1)$, which is orthogonal to the hyperplane~$\cW$. (Section~\ref{AC section} contains precise definitions of all the terminology in this paragraph.)

Similarly, our assertion that $\mu_k^S$ is absolutely continuous with respect to Lebesgue measure on~$\R^r$ would be true if we strengthened the hypothesis in Theorem~\ref{inclusive theorem}(a), by demanding that the set of vectors in equation~\eqref{Re and Im vectors} spanned $\R^r$ even when the constant vector $(1,\dots,1)$ was removed. However, given the current hypothesis, it might only be the case that $\mu_k^S$ is absolutely continuous with respect to Lebesgue measure on a hyperplane not containing $(1,\dots,1)$. Again this is sufficient, however, as such a hyperplane still intersects all wedges similar to~$\cS$, which are invariant under shifts in the direction $(1,\dots,1)$. The same comments apply to our statements that $\mu^R_W$ and $\mu^R$ are supported on all of~$\R^r$. The complication of including the constant vector $(1,\dots,1)$ in the statements of our theorems is necessary when $r=\phi(q)$; for smaller values of $r$, the complication is not necessary, but it does result in a weaker hypothesis and thus a stronger theorem.

\subsection{Organization of this paper}

We now summarize the contents of the remainder of this paper by section, including pointers to the most important auxiliary results; the discussion will also briefly introduce notation for the most prominent objects of study, which are fully defined as they arise in the argument.

In Section~\ref{explicit formulae section} we describe the traditional explicit formulas for the $r$-dimensional error term $E(x)$ defined in equation~\eqref{Ey}, including a version thereof (denoted by $E_T(x)$) where sums over zeros of the relevant $L(s,\chi)$ are truncated at height~$T$. Near the end of that section, we comment on prime number races involving the weighted counting functions $\theta(x;q,a)$ and $\psi(x;q,a)$ in place of $\pi(x;q,a)$.
We show in Section~\ref{convolution section} that the limiting logarithmic distribution of $E_T(x)$ can be decomposed, in a general way, as a convolution of two measures~$\mu_T^A$ and~$\mu_T^B$, where the contribution from an arbitrary prescribed set of self-sufficient zeros of robust characters is captured by~$\mu_T^A$ and separated from the contribution~$\mu_T^B$ of the remaining zeros. We show directly, using standard theorems from probability, that $\mu_T^A$ has a limiting distribution $\mu^A$ as $T\to\infty$, and subsequently that the limiting logarithmic distribution~$\mu$ of~$E(x)$ can itself be written as a convolution of this first limiting distribution $\mu^A$ and a second (less concrete) distribution $\mu^B$.
We also establish an upper bound for the mass that $\mu^B$ assigns to vectors of large norm from an analogous known estimate for~$\mu$ itself.
Based on this general decomposition, we then define specific random variables $X^R_W$ and $X^S_k$, using the self-sufficient zeros of $W$-robust and $k$-sturdy characters, respectively, whose probability distributions~$\mu^R_W$ and~$\mu^S_k$ are related to~$\mu$.

In Section~\ref{AC section} we explicitly describe the subspace $\cV_k^S$ of $\R^r$ that is the support of $\mu_{k}^S$ (the limiting distribution corresponding to the self-sufficient zeros of $k$-sturdy characters). We further show that $\mu_{k}^S$ is absolutely continuous with respect to Lebesgue measure on this subspace $\cV_k^S$, by writing down the formula for its characteristic function and showing that it decays rapidly enough to be integrable over that subspace.

Theorem~\ref{inclusive theorem}(a), the assertion that logarithmic densities exist if there are enough $k$-sturdy characters, is established in
Section~\ref{log densities section}. First we show (using absolute continuity) that the distribution $\mu_k^S$ does not concentrate on any points of the hyperplanes forming the boundary of the wedge $\cS$ defined in equation~\eqref{wedge definition}. It follows that $\mu$ itself has the same property, as convolving $\mu_k^S$ with the second distribution $\mu_k^N$ can only further smooth the distribution (in a precise sense that we describe). We also explain the slightly stronger Theorem~\ref{two-way theorem}(a) in this section, as well as the analogue Theorem~\ref{pili theorem}(a) for~$\pi(x)$ itself.

The complementary parts~(b) and~(c) of Theorem~\ref{inclusive theorem}---the assertions that logarithmic densities of (wedges or) balls in $\R^r$ are positive if there are enough ($W$-)robust characters---are proved in Section~\ref{cylinder section}.
Under the hypotheses of Theorem~\ref{inclusive theorem}(b), we show that $\mu^R_W$ assigns positive mass to certain ``cylinders'' (see Definition~\ref{cylinder def}) parallel to $(1,\dots,1)$, where these cylinders fill more and more of~$\R^r$ as~$W$ grows; we deduce that $\mu^R_W$ gives mass to every wedge such as~$\cS$ when~$W$ is sufficiently large. Again, it follows that $\mu$ itself has this property, as convolving $\mu^R_W$ with the second distribution $\mu^N_W$ can only further expand the support. We also show (Proposition~\ref{need all these directions prop}) how the precise form of our hypotheses such as equation~\eqref{Re and Im vectors} results naturally from our approach. The stronger hypotheses of Theorem~\ref{inclusive theorem}(c) imply that $\mu^R_W$ assigns positive mass to balls rather than cylinders, and also allow us to take~$W$ arbitrarily large.
In addition to establishing Theorem~\ref{pili theorem}(b)--(c) by similar methods, we also include two variants (Theorem~\ref{pili theorem for APs}) for the races between $\pi(x;q,a)$ and either $\li(x)/\phi(q)$ or $\pi(x)/\phi(q)$.

Appendix~\ref{probability appendix} contains background facts on probability measures, including tight sequences, weak convergence, and absolute continuity. In particular, we state a slight strengthening (Theorem~\ref{tightnesscriterion}) of an existing criterion for establishing tightness of a sequence of probability measures, and make explicit (Lemma~\ref{density}) the connection between the integrability of a characteristic function and the existence of a density function.
Appendix~\ref{KW appendix} contains material related to the Kronecker--Weyl equidistribution theorem, including
a decomposition of the limiting subtorus as a direct sum of two smaller subtori, in the case that the coordinates of the defining vector can be partitioned into two relatively independent sets.

\section{Explicit formulae, random variables, and probability measures}  \label{explicit formula section}

In this section, we convert the problem of showing that a given race $\{ a_1, \ldots, a_r \}$ modulo $q$ 
is strongly or weakly inclusive into a problem about random variables and convolutions of probability measures.  The first step in this 
conversion is the use of ``explicit formulae" in the style of Riemann.

\subsection{Explicit formulae}  \label{explicit formulae section}

As in prior work on this subject, our analysis begins with the explicit formula for the counting function of primes in an arithmetic progression. This explicit formula can be phrased in terms of the error term $E(x;q,a)$ defined in equation~\eqref{Exqa} as follows~\cite[Lemma~2.1]{RS}:
for $x\ge2$,
\begin{equation}  \label{bringing in explicit formula}
  E(x;q,a) =  -c(q,a) - \sum_{\chi \ne \chi_0} \overline\chi(a)
  \sum_{ \substack{ \gamma\in\R \\ L(\frac12+i \gamma,\chi)=0}} \frac{x^{i\gamma}}{\frac{1}{2}+i \gamma} + O\bigg( \frac1{\log x} \bigg),
\end{equation}
where
\[
c(q,a)=-1+\# \{ 0 \le b \le q-1 \colon b^2 \equiv a \mod q  \}.
\]
Rubinstein and Sarnak showed, assuming only GRH (as we do throughout), that $E(x;q,a)$ has a limiting logarithmic distribution; indeed, they established such a result~\cite[Theorem~1.1]{RS} for the vector-valued function $E(x)$ defined in equation~\eqref{Ey}.

\begin{prop}  \label{Ex lld prop}
The function $E(x)$ has a limiting logarithmic distribution $\mu$, in the sense of Definition~\ref{log dist def}. In particular, equation~\eqref{mu defined here} holds for all bounded continuous functions~$f$.
\end{prop}

\noindent We emphasize to the reader that throughout the rest of the main body of the paper, the symbol~$\mu$ denotes this specific measure, which depends upon $q$ and $a_1,\dots,a_r$. (It reverts to denoting a generic measure in Appendix~\ref{probability appendix}.)

Truncated versions of these explicit formulae will also be important in our analysis. We define these truncations now, and also introduce some vector-based notation that will prove convenient in the arguments to come.

\begin{definition}  \label{ETx def}
Define $\bv_\chi  =  (\chi(a_1), \ldots, \chi(a_r))$, and set
\begin{equation}
  \label{bdefn}
 \bb = -\big( c(q,a_1), \ldots, c(q,a_r) \big)  -2\sum_{\chi \ne \chi_0} \big( \ord_{s=1/2} L(s,\chi) \big) \overline\bv_\chi.
\end{equation}
Also define $\bx_\chi = \Re \bv_\chi$ and $\by_\chi = \Im \bv_\chi$.

Furthermore, define
$\theta_{\gamma} = \text{arg}(\frac{1}{2}+i \gamma)$, so that $\frac{1}{2}+i\gamma = e^{i \theta_\gamma} \sqrt{\frac{1}{4}+\gamma^2}$. Finally, for any positive real number $T$, define
\begin{equation}  \label{ETysum}
E_T(x) =  \bb+  2 \Re \sum_{\substack{\chi\mod q \\ \chi\ne\chi_0}} \overline\bv_\chi \sum_{\substack{ 0<\gamma \le T \\ L(\frac{1}{2}+i \gamma,\chi)=0}} e^{-i\theta_\gamma} \frac{x^{i\gamma}}{\sqrt{\frac{1}{4}+\gamma^2}}.
\end{equation}
This function was called $E^{(T)}(x)$ in \cite{RS}; in addition to modifying the notation, we have also removed the contributions from any zeros at $s=\frac12$ from the sums, placing them instead into the constant vector~$\bb$.
\end{definition}

Despite the new notation, a comparison of equations~\eqref{bringing in explicit formula} and~\eqref{ETysum} confirms that $E_T(x)$ really is a truncation of $E(x)$, and indeed one can show~\cite[equations~(2.5) and~(2.6)]{RS} that $E(x) = E_T(x) + O\big( x^{1/2} T^{-1} \log^2T + 1/\log x \big)$. Rubinstein and Sarnak analyzed these truncations as well~\cite[Lemma~2.3]{RS}:

\begin{prop}  \label{quoting ET}
For every $T>0$, the function $E_T(x)$ has a limiting logarithmic distribution~$\mu_T$. Furthermore, the probability measures $\{\mu_T\colon T>0\}$ converge weakly to $\mu$.
\end{prop}

\noindent Rubinstein and Sarnak do not explicitly state that $\{\mu_T\colon T>0\}$ converges weakly to $\mu$, but that deduction is implicit in the proof of~\cite[Lemma~2.3]{RS}, and it also follows from the argument in~\cite[Theorem~2.9]{ANS}.

As mentioned in Remark~\ref{variants of pi remark}, all our results involving the prime counting functions $\pi(x;q,a)$ are equally valid if we replace every occurrence of $\pi$ with either $\theta$ or $\psi$ (and, where appropriate, replace every occurrence of $\li(x)$ with $x$); we are now in a position to justify this remark. A straightforward partial summation argument (as in the proof of \cite[Lemma~2.1]{RS}) shows that
\begin{equation*}
\bigg| E(x;q,a) - \frac1{\sqrt x} \big( \phi(q) \theta(x;q,a) - \theta(x) \big) \bigg| \ll \frac1{\log x}.
\end{equation*}
Comparing to equation~\eqref{bringing in explicit formula}, we see that this bound suffices to imply that all results in this paper that are true for $E(x;q,a)$ (and its vector-valued analogues) are also true for the difference $\frac1{\sqrt x} \big( \phi(q) \theta(x;q,a) - \theta(x) \big)$. Furthermore, it is easy to see (and also contained in the proof of \cite[Lemma~2.1]{RS}) that
\[
\frac1{\sqrt x} \big( \phi(q) \theta(x;q,a) - \theta(x) \big) - \frac1{\sqrt x} \big( \phi(q) \psi(x;q,a) - \psi(x) \big) = -c(q,a) + O\bigg( \frac1{\log x} \bigg).
\]
The presence of the constant $-c(q,a)$ on the right-hand side causes some superficial changes, particularly in the definition~\eqref{bdefn} of $\bb$ in Definition~\ref{ETx def}: the term $-\big( c(q,a_1), \ldots, c(q,a_r) \big)$ would disappear if we switched from $\theta$ to $\psi$. However, the exact value of this constant vector $\bb$ has no effect on our arguments, and thus all of the proofs remain valid with $\frac1{\sqrt x} \big( \phi(q) \psi(x;q,a) - \psi(x) \big)$ in place of $E(x;q,a)$ as well. Similar remarks hold for the error terms relevant to Theorems~\ref{pili theorem} and~\ref{pili theorem for APs}.

\subsection{Separating the zeros into two sets}  \label{convolution section}

We now begin the process of dealing with our assumption that only some of the positive ordinates of nontrivial zeros of the Dirichlet $L$-functions\mod q are self-sufficient, that is, not involved in linear relations over the rationals with the other ordinates. We will partition the multiset of ordinates of zeros into two multisets, one containing many of the self-sufficient ordinates and the other containing the remainder of the ordinates. A key fact used in the argument to follow is that these two sets of ordinates are ``relatively independent'' over the rational numbers:

\begin{definition} \label{relind}
Let $0\le k_1\le k$ be integers. The finite multisets of real numbers $\{ \xi_1, \dots, \xi_{k_1} \}$ and $\{\xi_{k_1+1}, \dots, \xi_k\}$ are {\it relatively independent (over $\Q$)} if the following statement holds: whenever $\alpha_1,\dots,\alpha_k$ are rational numbers such that $\alpha_1 \xi_1 + \cdots + \alpha_k \xi_k = 0$, then both $\alpha_1 \xi_1 + \cdots + \alpha_{k_1} 
\xi_{k_1} = 0$ and $\alpha_{k_1+1} \xi_{k_1+1} + \cdots + \alpha_{k} 
\xi_{k} = 0$.
\end{definition}

\begin{remark} \label{stronger remark}
It is easy to check that this property of relative independence is preserved under taking subsets, and under multiplying all elements of both multisets by a nonzero real constant. It is also easy to check that if $\{ \xi_1, \dots, \xi_{k_1} \}$ and $\{\xi_{k_1+1}, \dots, \xi_k\}$ are relatively independent, and both multisets $\{ \xi_1, \dots, \xi_{k_1} \}$ and $\{\xi_{k_1+1}, \dots, \xi_k\}$ are individually linearly independent, then their union $\{ \xi_1, \dots, \xi_k\}$ is also linearly independent; more generally, if $d_1$, $d_2$, and $d$ are the dimensions of the $\Q$-vector spaces spanned, respectively, by $\{ \xi_1, \dots, \xi_{k_1} \}$, $\{\xi_{k_1+1}, \dots, \xi_k\}$, and $\{ \xi_1, \dots, \xi_k\}$, then $d=d_1+d_2$.
\end{remark}

\begin{remark} \label{ss and li and ri}
Let $\cG$ be a submultiset of $\Gamma(q)$. It is easy to check from Definitions~\ref{self-sufficient def} and~\ref{relind} that every element of $\cG$ is self-sufficient if and only if both $\cG$ is linearly independent and $\cG$ and $\Gamma(q)\setminus \cG$ are relatively independent.
\end{remark}

Our first goal is to show, given any set $\cG$ of self-sufficient ordinates of zeros, that the probability measure $\mu_T$  defined in Proposition~\ref{quoting ET} is a convolution of two related probability measures $\mu^A_{T}$ and $\mu^B_{T}$. We also compute the characteristic function of~$\mu^A_T$, which is constructed from the self-sufficient ordinates in $\cG$ less than~$T$; similarly, $\mu^B_T$ is constructed from the remaining ordinates less than~$T$.

\begin{definition} \label{G def}
Let $\cG$ be a subset of $\Gamma(q)$ such that every element of $\cG$ is self-sufficient. For $\gamma\in\cG$, let $\chi_\gamma$ denote the nonprincipal character\mod q for which $L(\frac12+i\gamma,\chi_\gamma)=0$ (this character is unique because $\gamma$ is self-sufficient).
Further, let $\bv_\gamma$ denote the vector $\bv_{\chi_\gamma}$.
\end{definition}

\begin{lemma} \label{nuTconvolution}
For $T >0$, let $\mu_T$ be the probability measure arising in Proposition~\ref{quoting ET}. There exist probability measures $\mu_T^A$ and $\mu_T^B$ such that:
\begin{enumerate}
 \item $\mu_{T} = \mu^A_{T} * \mu^B_{T}$;
 \item for $\bt \in \R^r$,
 \begin{equation}
   \label{ftformula}
   \widehat\mu^A_T(\bt) 
   = \prod_{\substack{\gamma\in\cG \\ 0<\gamma \le T}}  J_0 \bigg( \frac{2 |\bt \cdot \bv_\gamma| }{\sqrt{\frac{1}{4}+\gamma^2} }   \bigg).
 \end{equation}
\end{enumerate}
\end{lemma}

\begin{proof}
Define $k = \# \big( \Gamma(q) \cap (0,T] \big)$, that is, the number of zeros (counting multiplicity) up to height $T$ of all nonprincipal Dirichlet $L$-functions modulo~$q$. Also define $k_1 = \#(\cG \cap (0,T] \big)$
to be the number of ordinates $\gamma$ on the right-hand side of equation~\eqref{ftformula}. 
Let $\zeta^A$ denote a variable taking values in the $k_1$-dimensional torus, whose coordinates $\zeta_\gamma$ are indexed by those ordinates. Similarly, let 
$\zeta$ denote a variable taking values in the $k$-dimensional torus, whose coordinates are indexed by the multiset $\Gamma(q) \cap (0,T]$, and let $\zeta^B$ denote a variable taking values in the $(k-k_1)$-dimensional torus, whose coordinates are indexed by the remaining ordinates (those not in~$\cG$). There is a diagonal embedding $\Delta(t)$ of $\R$ into the $k$-dimensional torus, where the real number $t$ is replaced by the vector $\zeta$ for which each coordinate $\zeta_\gamma$ has been replaced by $\frac{t\gamma}{2\pi}$; there are similar diagonal embeddings $\Delta_1(t)$ of $\R$ into the $k_1$-dimensional torus and $\Delta_2(t)$ into the $(k-k_1)$-dimensional torus.

Now define
\begin{align*}
\Psi_1(\zeta^A) &= 2 \Re \bigg( \sum_{\substack{\gamma\in\cG \\ 0<\gamma\le T}}
\frac{e^{2 \pi i \zeta_\gamma}}{\frac{1}{2}+i \gamma} \overline\bv_\gamma \bigg) \\
\Psi_2(\zeta^B) &= 2 \Re \bigg( \sum_{\substack{\gamma\in\Gamma(q)\setminus\cG \\ 0<\gamma\le T}}
\frac{e^{2 \pi i \zeta_\gamma}}{\frac{1}{2}+i \gamma} \overline\bv_\gamma \bigg) \\
\Psi(\zeta) &= 2 \Re \bigg( \sum_{\substack{\gamma\in\Gamma(q) \\ 0<\gamma\le T}}
\frac{e^{2 \pi i \zeta_\gamma}}{\frac{1}{2}+i \gamma} \overline\bv_\gamma \bigg),
\end{align*}
and further define corresponding functions from $\R$ to $\R^r$:
\begin{align*}
  \eta_1(t) = \Psi_1\big( \Delta_1(t) \big), \quad \eta_2(t) = \Psi_2\big( \Delta_2(t) \big), \quad \eta(t) = \Psi\big( \Delta(t) \big).
\end{align*}
By Remark~\ref{ss and li and ri}, the set of ordinates $\gamma$ appearing in $\eta_1(t)$ and the set of ordinates $\gamma$ appearing in $\eta_2(t)$ are relatively independent sets, and this property is preserved when dividing all ordinates by~$2\pi$. This relative independence will be the crucial property that allows us to decompose $\mu_{T}$ as the convolution $\mu^A_{T} * \mu^B_{T}$.

By Corollary~\ref{kwcorollary} (allowing for the different conventions for indexing the variables), the functions $\eta_1(t)$, $\eta_2(t)$, and $\eta(t)$ possess limiting distributions $\nu_1$, $\nu_2$, and $\nu$, respectively; and by Lemma~\ref{kwconvolution} we know that $\nu=\nu_1*\nu_2$.
On the other hand, an examination of our notation, particularly that of Definition~\ref{ETx def}, reveals that $E_T(x) = \bb + \eta(\log x)$; it follows from Proposition~\ref{quoting ET} that $\mu_T$ is simply the translation of $\nu$ by the vector $\bb$, that is, $\mu_T(\cB) = \nu(\cB-\bb)$ for all Borel sets $\cB\subset \R^r$.
Define $\mu_T^B$ to be the translation of $\nu_2$ by the vector $\bb$, and simply define $\mu_T^A = \nu_1$; then it is trivial to check that $\mu_{T} = \mu^A_{T} * \mu^B_{T}$, establishing part~(a).

As for part (b), we begin by observing that by part (a) and the definition of $\mu^A_{T}$, Corollary~\ref{kwcorollary} implies that there exists a subtorus $\cA$ of $\mathbb{T}^{k_1}$ such that 
\begin{equation}  \label{proving part b}
  \lim_{y \to \infty} \frac{1}{y} \int_{0}^{y}  f(\eta_1(t)) \,dt = \int_{\cA} (f \circ  \Psi_1)(\ba) \,d \ba 
  = \int_{\R^r} f(\bx) \,d \mu^A_{T}(\bx)
\end{equation}
for all bounded continuous functions~$f$. Since the set $\frac1{2\pi} \cG$ is linearly independent by Remark~\ref{ss and li and ri}, it follows from Lemma~\ref{kw} that $\cA =\mathbb{T}^{k_1}$ and $d\ba = d\zeta^A$, which denotes Haar measure on $\T^{k_1}$.
Choosing $f(\bx) = e^{i \bt \cdot \bx}$, we see from equations~\eqref{charfcn} and~\eqref{proving part b} that 
\begin{equation}
   \label{ftdefinition}
   \widehat\mu^A_T(\bt) 
   = \int_{\R^r}  e^{i \bt \cdot \bx} \,d\mu^A_T(\bx)
   = \int_{\T^{k_1}} e^{i   \bt  \cdot  \Psi_1(\zeta^A)} \,d \zeta^A.
\end{equation}
Since
\[
\bt\cdot \Psi_1(\zeta^A) = \bt\cdot 2 \Re \bigg( \sum_{\substack{\gamma\in\cG \\ 0<\gamma\le T}} 
\frac{e^{2 \pi i \zeta_\gamma}}{\frac{1}{2}+i \gamma} \overline\bv_\gamma \bigg) = \sum_{\substack{\gamma\in\cG \\ 0<\gamma\le T}} 
\Re \bigg( \frac{2 \bt\cdot \overline\bv_\gamma}{\frac{1}{2}+i \gamma}e^{2 \pi i \zeta_\gamma} \bigg),
\]
we see that
\begin{align*}
\widehat\mu^A_T(\bt) 
&= \int_{\T^{k_1}} \exp \bigg( i \sum_{\substack{\gamma\in\cG \\ 0<\gamma\le T}} \Re \bigg( \frac{2 \bt\cdot \overline\bv_\gamma}{\frac{1}{2}+i \gamma}e^{2 \pi i \zeta_\gamma} \bigg) \bigg) \,d \zeta^A 
= \prod_{\substack{\gamma\in\cG \\ 0<\gamma\le T}} \int_{\T} \exp \bigg( i \Re \bigg( \frac{2 \bt\cdot \overline\bv_\gamma}{\frac{1}{2}+i \gamma}e^{2 \pi i \zeta_\gamma} \bigg) \bigg) \,d \zeta_\gamma.
\end{align*}
Using the identity $J_0(|z|)=\int_{0}^{1}  e^{i \Re( z e^{2 \pi i \theta})} \,d \theta$, we conclude that
\begin{align*} 
 \widehat\mu^A_T(\bt) &= \prod_{\substack{\gamma\in\cG \\ 0<\gamma\le T}} J_0 \bigg( \bigg| \frac{2 \bt\cdot \overline\bv_\gamma}{\frac{1}{2}+i \gamma} \bigg| \bigg) = \prod_{\substack{\gamma\in\cG \\ 0<\gamma\le T}} J_0 \bigg( \frac{2 |\bt\cdot \bv_\gamma|}{\sqrt{\frac{1}{4}+\gamma^2}} \bigg)
\end{align*}
as desired.
\end{proof}

The computation in part~(b) of the above proof has been done by various authors in various ways---see for example~\cite[Theorem~1.9]{ANS} and~\cite[Section 3.1]{RS}.

\subsection{Random variables and convolutions}  \label{sums of rvs section}

From these results concerning the measure $\mu_T$ associated with the truncated error term, we can deduce analogous results about the limiting logarithmic distribution $\mu$ given in Proposition~\ref{Ex lld prop}. In particular, we prove in this section that $\mu$ is itself the convolution of two probability measures $\mu^A$ and $\mu^B$.

In order to study the probability measure $\mu$, it is convenient to interpret $\mu$ in terms of a vector-valued 
random variable. 
We can guess which random variable to study from the explicit formula~\eqref{ETysum}.  Note that if LI is true, then
the values of the functions  $e^{i (y \gamma-\theta_{\gamma})}$ behave in the limit like the values of 
$\{ Z_{\gamma} \}$, a sequence of independent random variables, where each $Z_\gamma$ is a random variable uniformly distributed on the unit circle in $\C$.  In fact, assuming LI, one can deduce (from Lemma~\ref{nuSTweakconv} below, for example) that $\mu$ equals the probability measure
associated to the vector-valued random variable
\begin{equation*}
X  = \bb + 2 \Re \sum_{\substack{\chi\mod q \\ \chi\ne\chi_0}} \overline\bv_\chi \sum_{\substack{\gamma>0 \\ L(\frac12+i\gamma,\chi)=0}} \frac{Z_\gamma}{\sqrt{\frac14+\gamma^2}}.
\end{equation*}

However, in our setting we are not assuming the truth of LI. Heuristically, we are motivated by the idea that there should be independent random variables $X^A$ and $X^B$ such that $X = X^A+ X^B$, where $X^A$ is made from all of the (self-sufficient) zeros in $\cG$ and $X^B$ from all of the remaining zeros; in this situation, the distribution $\mu$ would decompose as $\mu=\mu^A*\mu^B$. For example, using the vectors $\bv_\gamma$ from Definition~\ref{G def}, we can define the random vector $X^A$ to be
\begin{equation}  \label{XA def}
X^A =  2 \Re \sum_{\gamma\in\cG} \frac{Z_\gamma \overline\bv_\gamma}{\sqrt{\frac14+\gamma^2}},
\end{equation}
where the $Z_\gamma$ appearing in the definition are jointly independent random variables, each uniformly distributed on the unit circle in~$\C$. Perhaps the definition of $X^B$ would look like
\begin{equation*}
X^B \underset?= \bb + 2 \Re \sum_{\substack{\chi\mod q \\ \chi\ne\chi_0}} \overline\bv_\chi \sum_{\gamma \in \Gamma(\chi) \setminus \cG} \frac{Z_\gamma}{\sqrt{\frac14+\gamma^2}},
\end{equation*}
where the dependences among the various $Z_\gamma$ are ``inherited'' from any $\Z$-linear relations among the various $\gamma$: if 
$c_1\gamma_1 + \dots + c_k\gamma_k = 0$,
then 
$(Z_{\gamma_1})^{c_1} \cdots (Z_{\gamma_k})^{c_k} = 1$. However, in addition to this description of the dependences being less precise than we would like, it is not even clear {\it a priori} that the potentially infinite sum in the proposed definition of $X^B$ converges almost surely.

For this reason, we take a different approach: we do define the random variable $X^A$ formally, exactly as in equation~\eqref{XA def}, which will give rise to a distribution $\mu^A$. Then, later (in Proposition~\ref{existence of muN prop}), we show that there exists a distribution $\mu^B$ such that $\mu=\mu^A*\mu^B$. Fortunately, we will need to know little about $\mu^B$ other than its mere existence.

\begin{lemma} \label{nuSTweakconv}
There exists a probability measure $\mu^A$ on $\R^r$ such that $\mu^A_{T} \to \mu^A$ weakly. 
\end{lemma}

\begin{proof}
Define $X^A$ as in equation~\eqref{XA def} above.
Observe that 
$X^A  = (Y_1, \ldots, Y_r )$
where 
\[
 Y_j = 2 \Re \sum_{\gamma\in\cG} \frac{\overline\chi_\gamma(a_j) Z_{\gamma}}{|\frac{1}{2}+i \gamma|}
\]
is a real-valued random variable defined on $\T^{\infty}$; indeed,
$Y_j=Y_j^\Re - Y_j^\Im$ where 
\[
    Y_j^\Re = 2 \sum_{\gamma\in\cG} \frac{\Re\overline\chi_\gamma(a_j) \cdot \Re Z_{\gamma}}{|\frac{1}{2}+i \gamma|}
      \text{ and }
    Y_j^\Im = 2 \sum_{\gamma\in\cG} \frac{\Im\overline\chi_\gamma(a_j) \cdot \Im Z_{\gamma}}{|\frac{1}{2}+i \gamma|}.
\]
Since the $Z_{\gamma}$ are independent, identically distributed random variables with values on the unit circle,
it follows that $\{ \Re Z_{\gamma}\}$ and $\{ \Im Z_{\gamma}\}$  are sets of independent, identically distributed random variables with values in $[-1,1]$.

The total number of zeros of Dirichlet $L$-functions\mod q in the critical strip whose imaginary parts are between 0 and $T$ is asymptotic to $\phi(q)(T\log qT)/2\pi$~\cite[Corollary 14.7]{MV}; it follows easily that
\begin{equation}  \label{pre yes it converges}
\sum_{\chi\mod q} 
 \sum_{\substack{ \gamma\in\Gamma(\chi)}} |\tfrac{1}{2}+i \gamma|^{-2} \text{ converges}.
\end{equation}
In particular, $\sum_{\gamma\in\cG} |\tfrac{1}{2}+i \gamma|^{-2}$ converges,
whence a theorem of Komolgorov and Khinchin \cite[Theorem~1, p. 384]{Sh} 
implies that each $Y_j^\Re$ converges with probability~$1$ and each $Y_j^\Im$ converges with probability~$1$.  
It follows that each $Y_j$ converges with probability 1, and therefore
$X^A$ converges almost surely.

Let $\mu^A$ denote the probability measure on $\R^r$ associated to the random vector $X^A$, namely the pushforward (see Definition~\ref{pushforward def}) of Haar measure on $\T^{\infty}$ under the map given by the right-hand side of equation~\eqref{XA def}.
A computation similar to Lemma~\ref{nuTconvolution}(b) shows that 
\begin{equation} 
      \label{muhatR}
     \widehat\mu^A(\bt) 
   = \prod_{\gamma\in\cG}  J_0 \bigg( \frac{2 |\bt \cdot \bv_\gamma| }{\sqrt{\frac{1}{4}+\gamma^2} }   \bigg).
 \end{equation}
 Observe that since $J_0(z) = 1+O(|z|^2)$ for $|z| \le 1$, the convergence of
 $\sum_{\gamma\in\cG}  (\frac{1}{4} + \gamma^2)^{-1}$ implies that the above infinite product converges absolutely.
It follows from equation~\eqref{ftformula} 
that $\widehat\mu^A_T(\bt) \to \widehat\mu^A(\bt)$ for all $\bt \in \R^r$ as $T \to \infty$.  
By Levy's theorem \cite[p. 383]{Bi}, $\mu^A_{T} \to \mu^A$ weakly as $T \to \infty$. 
\end{proof}

We remark that Lemma~\ref{nuSTweakconv} has no assumption about the size of~$\cG$. Its statement is valid even if $\cG$ is empty, for example: in that case, the random variable $X^A$ defined in equation~\eqref{XA def} would be identically~$0$, relevant characteristic functions such as the one in equation~\eqref{ftformula} would be identically~$1$, and the corresponding measure $\mu^A$ would be a Dirac delta measure. On the other hand, we will apply the lemma only in situations where $\cG$ contains plenty of self-sufficient zeros. Similar comments apply to the next two propositions.

\begin{prop}  \label{existence of muN prop}
There exists a probability measure $\mu^B$ such that $\mu = \mu^A*\mu^B$. 
\end{prop}

\begin{proof}
Define
\[
\cN = \bigg\{ \bt = (t_1,\dots,t_r) \in \R^r \colon \max_{1\le j\le r} |t_j| < \frac3{5r} \bigg\}.
\]
Note that for any zero $\rho$ of any $L(s,\chi)$ (assuming GRH),
\[
\frac{2}{|\rho|} \bigg| \sum_{j=1}^{r} t_j \overline\chi(a_j) \bigg| \le 4\bigg| \sum_{j=1}^{r} t_j \overline\chi(a_j) \bigg| \le 4r\max_{1\le j\le r} |t_j| < \frac{12}5
\]
for all $\bt\in\cN$. The least positive root of the Bessel function $J_0(t)$ occurs at $t\approx 2.404$; in particular, the above inequality shows that $J_0\big( \frac{2}{|\rho|} \big| \sum_{j=1}^{r} t_j \overline\chi(a_j) \big| \big) \ne0$ for all $\bt\in\cN$. We conclude from equation~\eqref{ftformula} that $\widehat\mu^A_T(\bt) \ne 0$ for $\bt \in \cN$, independent of the value of $T \ge 0$.

Next, we prove that the family of probability measures $(\mu^B_T)_{T > 0}$ is tight.
By Lemma~\ref{nuTconvolution}(a), $\widehat\mu_{T}(\bt) = \widehat\mu^A_T(\bt) \widehat\mu^B_T(\bt)$
for all $\bt \in \R^r$; when $\bt \in \cN$, we may divide by $\widehat\mu^A_T(\bt)$ to obtain
$\widehat\mu^B_T(\bt) =  \widehat\mu_{T} (\bt) /  \widehat\mu^A_T(\bt)$.  
Since $\mu_T \to \mu$ weakly by Proposition~\ref{quoting ET} and $\mu^A_T \to \mu^A$ weakly by Lemma~\ref{nuSTweakconv}, it follows from Levy's theorem that
\[
h(\bt) := \lim_{T \to \infty} \widehat\mu^B_T(\bt) =  \lim_{T \to \infty} \frac{\widehat\mu_T(\bt)}{\widehat\mu^A_T(\bt)} = \frac{\widehat\mu(\bt)}{\widehat\mu^A(\bt)}
\]
exists for $\bt \in \cN$. (The argument that $\widehat\mu^A(\bt)\ne0$ is the same as the argument above showing $\widehat\mu^A_T(\bt)\ne0$, using the formula~\eqref{muhatR} in place of~\eqref{ftformula}.)
By \cite[Theorem~1, p. 278]{Sh}, both $\widehat\mu(\bt)$ and $\widehat\mu^A(\bt)$ are continuous at $\bt={\bf 0}$; and so $h(\bt)$ is also continuous at $\bt= {\bf 0}$.
Thus, by Theorem~\ref{tightnesscriterion}, $(\mu^B_T)_{T > 0}$ is tight.

By Theorem~\ref{tightsubsequence} 
there exists a subsequence $( \mu^B_{T_k} )_{k \in \mathbb{N}}$ and a probability measure $\mu^B$ such
that $\mu^B_{T_k} \to \mu^B$ weakly.   By Lemma~\ref{weakconv}, it follows that 
$\mu^A_{T_k}*\mu^B_{T_k} \to \mu^A*\mu^B$
weakly.   On the other hand, $\mu^A_{T_k}*\mu^B_{T_k} = \mu_{T_k}$ by Lemma~\ref{nuTconvolution}(a) 
and $\mu_{T_k} \to \mu$ weakly by Proposition~\ref{quoting ET}.  Combining these facts,
\[
  \int_{\R^r} f(\bx) \,d (\mu^A*\mu^B)(\bx)=  \int_{\R^r} f(\bx) \,d\mu(\bx) 
\]
for all bounded continuous $f(x)$ on $\R^r$, and thus $\mu=\mu^A*\mu^B$ by 
\cite[Theorem~1.2]{Bi2}.
\end{proof}

\begin{remark} \label{muA muB on hyperplane}
When $r=\phi(q)$, the measures $\mu^A$ and $\mu^B$ described in Lemma~\ref{nuSTweakconv} and Proposition~\ref{existence of muN prop} are actually supported on the hyperplane $x_1+\cdots+x_{\phi(q)}=0$. This holds for $\mu^A$ because the random variable $X^A$ from equation~\eqref{XA def} is supported on this hyperplane, which in turn follows from the fact that for each nonprincipal character~$\chi$, the vector $\bv_\chi$ lies in the hyperplane by orthogonality. Then, the fact that $\mu$ itself is supported on this hyperplane (as mentioned in Section~\ref{outline section}) implies that $\mu^B$ is also supported on the hyperplane, since the support of $\mu=\mu^A*\mu^B$ is the set sum of the supports of $\mu^A$ and $\mu^B$. (While this argument is self-contained, it could be reassuring to note that the constant vector $\bb$ defined in Definition~\ref{ETx def} does lie in this hyperplane when $r=\phi(q)$.)
\end{remark}

Two special cases of the general results above will be crucial for our proofs of Theorems~\ref{all-way theorem} and~\ref{inclusive theorem}. The following proposition is relevant for part~(a) of each theorem, while the proposition after it is relevant for parts~(b) and~(c) of each theorem.

\begin{prop}  \label{muconvolutionsturdy}
For any positive integer $k$, there exist probability measures $\mu^S_k$ and $\mu^N_k$ such that $\mu = \mu_k^S*\mu^N_k$ and
 \begin{equation}
  \label{muhatn0S}
  \widehat\mu_{k}^{S}(\bt) = \prod_{\substack{\chi\mod q \\ \chi\ne\chi_0 \\ \chi \ k\text{\rm-sturdy}}} \prod_{\gamma  \in \Gamma^{S}(\chi)} J_0 \bigg( \frac{2 |\bt \cdot \bv_\chi| }{\sqrt{\frac{1}{4}+\gamma^2} }   \bigg). 
\end{equation}
\end{prop}
 
Indeed, $\mu_k^S$ can be defined as the distribution of the random variable
\begin{equation}
  \label{muSdefn}
X_k^S = 2 \Re \sum_{\substack{\chi\mod q \\ \chi\ne\chi_0 \\ \chi \text{ $k$-sturdy}}} \overline\bv_\chi \sum_{\gamma\in\Gamma^S(\chi)} \frac{Z_\gamma}{\sqrt{\frac14+\gamma^2}}
\end{equation}
where we stipulate that $\{Z_\gamma\}$ is an independent collection of random variables, each uniformly distributed on the unit circle in~$\C$; however, the characteristic function~\eqref{muhatn0S} is essentially all we will need going forward.

\begin{prop}  \label{muconvolutionrobust}
For any positive real number $W>0$, there exist probability measures $\mu^R_W$ and $\mu^N_W$ such that $\mu = \mu^R_W*\mu^N_W$ and
 \begin{equation}
  \label{muhatn0SB}
  \widehat\mu^R_W(\bt) = \prod_{\substack{\chi\mod q \\ \chi\ne\chi_0 \\ \chi \text{ \rm $W$-robust}}} \prod_{\gamma  \in \Gamma^{S}(\chi)} J_0 \bigg( \frac{2 |\bt \cdot \bv_\chi| }{\sqrt{\frac{1}{4}+\gamma^2} }   \bigg). 
\end{equation}
\end{prop}
 
In this case also, $\mu^R_W$ can be defined as the distribution of the random variable
\begin{equation}
  \label{XR def}
X^R_W = 2 \Re \sum_{\substack{\chi\mod q \\ \chi\ne\chi_0 \\ \chi \text{ $W$-robust}}} \overline\bv_\chi \sum_{\gamma\in\Gamma^S(\chi)} \frac{Z_\gamma}{\sqrt{\frac14+\gamma^2}}
\end{equation}
where we stipulate that $\{Z_\gamma\}$ is an independent collection of random variables, each uniformly distributed on the unit circle in~$\C$.

We will need one more fact regarding these decompositions of the form $\mu=\mu^A*\mu^B$, namely that the mysterious probability measure $\mu^B$ has mass somewhere sufficiently near the origin in~$\R^r$. We can accomplish this task by using a multidimensional version of Chebyshev's inequality.

Given any probability measure~$\mu_1$ on~$\R^r$, we can write $\E(\mu_1) = \int_{\R^r} \bu \,d\mu_1(\bu)$ and
\[
\sigma^2(\mu_1) 
= \int_{\R^r} \|\bu\|^2 \,d\mu_1(\bu) - \bigg\| \int_{\R^r} \bu \,d\mu_1(\bu) \bigg\|^2.
\]
We write $\sigma(\mu_1) = \sqrt{\sigma^2(\mu_1)}$ (in contrast to some usage where $\sigma(\mu_1)$ is an $\R^r$-valued quantity).
In this setting, Chebyshev's inequality is then valid (see~\cite[Section~4]{Fe} for the equivalent statement for random variables) in the form
\begin{equation} \label{Rr Chyeyshev}
\mu\big( \big\{ x\in\R^r\colon |x - \E(\mu_1)| \ge \lambda\sigma(\mu_1) \big\} \big) \le \frac1{\lambda^2}.
\end{equation}
When two measures are related via a convolution, we can use this inequality to compare the masses the measures assign to complements of balls.

\begin{definition} \label{ball def}
Let $\cB_{\rho}(\bx_0)$ denote the ball in $\mathbb{R}^r$ with radius $\rho$ and center $\bx_0$, with the shorthand $\cB_\rho = \cB_\rho(\bzero)$ for balls centered at the origin; we further define $\cB^c_\rho = \{ x \in \R^r  \colon \|x\| \ge \rho \}$ to be the complement of~$\cB_\rho$.
\end{definition}

\begin{lemma} \label{tail bounds for convolution}
Let $\mu_1$ and $\mu_2$ be any probability measures on $\R^r$ and set $\mu_0=\mu_1*\mu_2$. Define $V_0 = 2\sigma(\mu_1) + \|\E(\mu_1)\|$. Then for $V >V_0$,
\[
  \mu_2(\cB^c_{V}) \ll \mu_0(\cB^c_{V-V_0}). 
\]
\end{lemma}

\begin{proof}
Note that if both $\|\bx-\E(\mu_1)\|<2\sigma(\mu_1)$ and $\|\by\|\ge V$ hold, then by the triangle inequality, $\|\bx+\by\| > V-2\sigma(\mu_1)-\|\E(\mu_1)\| = V-V_0$. In other words, $\cB_{2\sigma(\mu_1)}(\E(\mu_1)) + \cB^c_V \subset \cB^c_{V-V_0}$, and
consequently $\mu_1\big( \cB_{2\sigma(\mu_1)}(\E(\mu_1)) \big) \mu_2(\cB^c_V) \le \mu_0(\cB^c_{V-V_0})$ by equation~\eqref{convolution lower bound}.
On the other hand, equation~\eqref{Rr Chyeyshev} with $\lambda=2$ implies that $\mu_1\big( \cB_{2\sigma(\mu_1)}(\E(\mu_1)) \big) = 1 - \mu_1\big( \cB^c_{2\sigma(\mu_1)}(\E(\mu_1)) \big) \ge \frac34$, which establishes the lemma.
\end{proof}

We will apply the following proposition in Section~\ref{cylinder section}, in the proofs of Theorem~\ref{inclusive theorem}(b)--(c) with $\mu^B = \mu^N_W$ and in the proofs of Theorem~\ref{pili theorem}(b)--(c) with $\mu^B$ equal to a measure $\mu^N_{\pi}$ related to the race between $\pi(x)$ and $\li(x)$.

\begin{prop} \label{has mass somewhere small}
There exists a constant $\beta(q)>0$ such that, for any decomposition $\mu=\mu^A*\mu^B$ of the type considered in this section, and any $\rho>0$, there exists $\bn\in\R^r$ with $\|\bn\| \le \beta(q)$ such that $\mu^B(\cB_\rho(\bn)) > 0$. Moreover, if $r=\phi(q)$ then~$\bn$ is contained in the hyperplane $x_1+\cdots+x_{\phi(q)}=0$.
\end{prop}

\begin{proof}
We apply Lemma~\ref{tail bounds for convolution} with $\mu_1$ and $\mu_2$ equal to $\mu^A$ and~$\mu^B$, so that $\mu_0 = \mu^A*\mu^B = \mu$.
By equation~\eqref{XA def}, $\mu_A$ is the distribution of the random vector
\begin{equation*}
X^A =  2 \Re \sum_{\gamma\in\cG} \frac{Z_\gamma \overline\bv_\gamma}{\sqrt{\frac14+\gamma^2}},
\end{equation*}
where the $Z_\gamma$ are independent random variables with mean~$0$ and variance~$\frac12$; therefore the variance of~$\mu_1$ equals
\[
4 \sum_{\gamma\in\cG} \frac{\frac12 \|\bv_\gamma\|^2}{\frac14+\gamma^2} = 2r \sum_{\gamma\in\cG} \frac1{\frac14+\gamma^2} \ll \phi(q) \sum_{\gamma\in\Gamma(q)\cup \Gamma(\chi_0)} \frac1{\frac14+\gamma^2} \ll_q 1.
\]
Therefore by Lemma~\ref{tail bounds for convolution}, there exists a constant $V_0$ depending on~$q$ such that $\mu^B(\cB^c_{V}) \ll \mu(\cB^c_{V-V_0})$ when $V>V_0$.

Rubinstein and Sarnak prove~\cite[Theorem~1.2]{RS}, assuming GRH, that there exist positive constants $c_1$ and $c_2$ depending on~$q$
such that $\mu = \mu_{q;a_1, \ldots, a_r}$ satisfies $\mu(\cB^c_{V}) \le c_1 \exp(-c_2 \sqrt{V})$; therefore for $V>V_0$,
\[
  \mu^B(\cB^c_{V}) \ll c_1 \exp(-c_2 \sqrt{V-V_0}).
\]
In particular, there exists $\beta(q)$ such that $\mu^B(\cB^c_{\beta(q)}) \le \frac12$, which shows that $\mu^B(\cB_{\beta(q)}) \ge \frac12$. But $\cB_{\beta(q)}$ can be covered by finitely many balls of the form $\cB_\rho(\bn)$, so one such ball must be assigned positive measure by~$\mu^B$.

The fact that~$\bn$ is contained in the hyperplane $x_1+\cdots+x_{\phi(q)}=0$ when $r=\phi(q)$ follows from the fact that $\mu^B$ is supported on this hyperplane, as shown in Remark~\ref{muA muB on hyperplane}.
\end{proof}

\section{$\mu_{k}^S$ is absolutely continuous with respect to Lesbesgue measure}  \label{AC section}

In this section we show that $\mu_{k}^S$ is absolutely continuous with respect to Lebesgue measure
of a certain subspace $\cV_k^S$. (This subspace will be defined in Definition~\ref{xyz vectors definition}; the notion of Lebesgue measure on a subspace is clarified in Definition~\ref{lambda A definition} below.) This absolute continuity is crucial to the proof of Theorem~\ref{inclusive theorem}(a) and its relatives, which assert the existence of the logarithmic densities associated with prime number races. Establishing this absolute continuity requires a bound on the characteristic function $\widehat\mu_{k}^S$, which is a product of Bessel functions each of which decays in a specific direction; showing that the product actually decays in all directions will involve the expression
 $  \sum_{\chi \ k\text{-sturdy}} |  \bv_\chi \cdot \bt |^2$. The following linear algebra argument gives a convenient lower bound for sums of this type.

\begin{lemma}  \label{quadform lemma}
Let $r$ and $m$ be positive integers. Let $\bv_1,\dots,\bv_m$ be vectors in $\C^r$, and define
\[
\cV = \spn \big\{ \Re \bv_1, \dots, \Re \bv_m, \Im \bv_1, \dots, \Im \bv_m \big\} \subset \R^r.
\]
Then $\| \bx \|^2 \ll \sum_{j=1}^m |\bv_j \cdot \bx|^2$ for all $\bx \in \cV$, where the implicit constant may depend on $\{\bv_1,\dots,\bv_m\}$.
\end{lemma}

\begin{proof}
Every spanning set contains a basis, so select a maximal $\R$-linearly independent subset
\[
\{ \bw_1, \dots, \bw_\ell \} \subset \big\{ \Re \bv_1, \dots, \Re \bv_m, \Im \bv_1, \dots, \Im \bv_m \big\},
\]
so that $\cV = \spn\{ \bw_1, \dots, \bw_\ell \}$.
Observe that for any real vector $\bx$,
\[
  |\bv_j \cdot \bx|^2 = \big| ( \Re\bv_j + i \Im\bv_j) \cdot \bx \big|^2
  = (\Re\bv_j \cdot \bx)^2 + (\Im\bv_j \cdot \bx)^2
\]
for each $1\le j\le m$; therefore
\begin{equation*} 
  \sum_{j=1}^m |\bv_j \cdot \bx|^2 = \sum_{j=1}^m \big( (\Re\bv_j \cdot \bx)^2 + (\Im\bv_j \cdot \bx)^2 \big) \ge \sum_{j=1}^{\ell} (\bw_j \cdot \bx)^2.
\end{equation*}
However, the square root of the right-hand side is a norm on~$\cV$, as is the restriction of $\| \bx \|$ to~$\cV$; since $\cV$ is a finite-dimensional vector space, these two norms are equivalent~\cite[Theorem~2.4-4]{Kr}, and thus in particular
\[
\| \bx \|^2 \ll \sum_{j=1}^{\ell} (\bw_j \cdot \bx)^2 \le \sum_{j=1}^m |\bv_j \cdot \bx|^2
\]
as desired.
\end{proof}

\begin{definition}  \label{xyz vectors definition}
Using the vectors $\bx_\chi$ and $\by_\chi$ from Definition~\ref{ETx def}, define the real vector space
\begin{equation}
  \label{Vdefn}
  \begin{split}
  \cV_k^S &=  \text{Span} ( \{ \bx_\chi, \by_\chi \colon \chi\mod q,\, \chi\ne\chi_0,\, \chi \text{ is $k$-sturdy} \} ),
  \end{split}
\end{equation}
which is a subspace of~$\R^r$. For example, if $\chi_1$ and $\chi_2$ are the only $k$-sturdy nonprincipal characters\mod q, then $\cV_k^S$ is spanned by the four vectors $\bx_{\chi_1}, \by_{\chi_1}, \bx_{\chi_2}, \by_{\chi_2}$.

Recall that a probability measure $\mu$ on $\R^r$ is {\it supported on} a subset $\cS$ if, for every $\bx\in\R^r\setminus\cS$, there exists $\ep>0$ such that $\mu(\cB_\ep(\bx))=0$.
It is apparent from 
equation~\eqref{muSdefn} (or, with a little thought, from equation~\eqref{muhatn0S}) that the probability measure $\mu^S_k$ is supported on~$\cV^S_k$.
\end{definition}

We now aim to show that $\mu_k^S$ is absolutely continuous with respect to $\lambda_{\cV_{k}^S}$.    A standard result in probability is that if $\mu$ is a probability measure on $\mathbb{R}^n$ and 
$\int_{\R^n} |\widehat{\mu}(\bt)| \,d \lambda_{\R^n}(\bt)$ converges, then $\mu$ is absolutely continuous
with respect to $\lambda_{\R^n}$.  We give the generalization of this statement in Proposition~\ref{abscontV} below, where we will show that if $\mu$ is supported on a subspace $\cV$ of $\R^n$ and 
\begin{equation}
  \label{integralconvergence}
\int_{ \cV} | \widehat\mu (\bt) | \,d \lambda_{\cV}(\bt) \text{ converges},
\end{equation} 
then    
$\mu$ is absolutely continuous with respect to $\lambda_{\cV}$.

Our first goal, consequently, is to establish the assertion~\eqref{integralconvergence} in the case
$\mu = \mu_k^S$ and $\cV= \cV_k^S$. Our strategy is to use equation~\eqref{muhatn0S} to obtain a pointwise bound  for $\widehat\mu^S_k(\bt)$. 

\begin{lemma} \label{muhatSbounds}
Define
\begin{equation}  \label{compact K definition}
\cK=  \{ \bt\in \cV_k^S  \colon | \bv_\chi \cdot \bt | \le 1 \text{ for every $k$-sturdy character } \chi \}.
\end{equation}
\begin{enumerate}
\item If $\bt\in \cV_k^S \setminus \cK$, then
\begin{equation}
   |\widehat\mu_{k}^S(\bt)|  \ll_{k}    \|\bt\|^{-k/2}.
\end{equation}
\item $\cK$ is bounded and contains a neighborhood of $\bzero$ in $\cV_k^S$.
\end{enumerate}
\end{lemma}
\begin{proof}
Fix $\bt\in \cV_k^S \setminus \cK$. Let $\chi_1,\dots,\chi_b$ be the $k$-sturdy characters such that $|\bv_\chi \cdot \bt| > 1$; note that $b\ge1$ since $\bt\notin\cK$.
We apply the Bessel function bound  
\begin{equation}
  \label{J0bound}
  |J_0(x)| \le \min \Big\{ 1, \sqrt{\tfrac{2}{\pi |x|}} \Big\}
\end{equation}
(see \cite[Theorem~7.31.2]{Sz}) to equation~\eqref{muhatn0S}. 
If $\Gamma^S_k(\chi)$ denotes a fixed set of $k$ self-sufficient ordinates of $\chi$, then
by \eqref{muhatn0S}
\begin{equation}  \label{applied Bessel bound}
\begin{split}
  |\widehat\mu_{k}^S(\bt)| &\le \prod_{j=1}^b \prod_{\gamma \in \Gamma^S_k(\chi_j)} \bigg| J_0 \bigg( \frac{ 2 |  \bv_{\chi_j} \cdot \bt | }{ \sqrt{\frac{1}{4}+\gamma^2}} \bigg) \bigg| \le \prod_{j=1}^b \prod_{\gamma \in \Gamma^S_k(\chi_j)} \frac{(\frac14+\gamma^2)^{1/4}}{\sqrt{\pi|\bv_{\chi_j}\cdot\bt|}} \\
  &\ll_{b,k} \prod_{j=1}^b |\bv_{\chi_j} \cdot \bt|^{-k/2}
  \ll_{b,k} \bigg( \sum_{j=1}^b |\bv_{\chi_j} \cdot \bt|^2 \bigg)^{-k/4},
\end{split}
\end{equation}
where the last inequality uses the elementary bound 
$\prod_{j=1}^b x_j \ge \frac{1}{b} \sum_{j=1}^b x_j$ for real numbers $x_j > 1$.
We further have
\[
   \frac1b \sum_{j=1}^b |\bv_{\chi_j} \cdot \bt|^2 \ge \frac1{\#\{\chi\mod q,\, \chi\ne\chi_0,\, \chi \text{ $k$-sturdy}\}} \sum_{\substack{\chi\mod q \\ \chi\ne\chi_0 \\ \chi \text{ $k$-sturdy}}} |\bv_\chi \cdot \bt|^2;
\]
since the average of the $b$ largest elements of $|\bv_\chi \cdot \bt|^2$ is greater than or equal to the average of all the numbers $|\bv_\chi \cdot \bt|^2$. As our $\ll$-constants may depend on $q$ (and thus are uniform in integers such as $b$ that must lie between $1$ and $\phi(q)$), we conclude that
\begin{equation*}
     |\widehat\mu_{k}^S(\bt)|  \ll_{k}  \bigg( \sum_{\substack{\chi\mod q \\ \chi\ne\chi_0 \\ \chi \text{ $k$-sturdy}}} |\bv_\chi \cdot \bt|^2 \bigg)^{-k/4}.
\end{equation*}
Finally, applying Lemma~\ref{quadform lemma} with $\cV = \cV_k^S$ and $\bv_1,\dots,\bv_m$ equalling the $\bv_\chi$ corresponding to the nonprincipal $k$-sturdy characters modulo $q$, we conclude that
\[
      |\widehat\mu_{k}^S(\bt)|  \ll_{k}    \|\bt\|^{-k/2}
      \text{ for } 
      \bt\in \cV_k^S,
\]
and the proof of part (a) is complete.

On the other hand, if $\bt\in\cK$, so that $|\bv_{\chi} \cdot \bt| \le 1$ for all nonprincipal $k$-sturdy characters modulo~$q$, then the same application of Lemma~\ref{quadform lemma} immediately shows that $\|\bt\|^2 \ll \sum_{\chi} |\bv_\chi\cdot\bt|^2 \ll 1$, which shows that $\cK$ is bounded. Moreover, $\cK$ contains the set
\[
\{ \bt\in \cV_k^S  \colon | \bv_\chi \cdot \bt | < 1 \text{ for every $k$-sturdy character } \chi \} = \bigcap_{\chi \text{ $k$-sturdy}} \{ \bt\in \cV_k^S  \colon | \bv_\chi \cdot \bt | < 1 \};
\]
each set in the intersection is the inverse image of the open interval $(-1,1)$ under the continuous map $\bt \mapsto \bv_\chi \cdot \bt$, and therefore the intersection is itself an open set, which clearly contains~$\bzero$. Thus the proof of part (b) is also complete.
\end{proof}

\begin{definition}  \label{lambda A definition}
Let $\cV$ be an $\ell$-dimensional subspace of $\R^n$. We define {\em $\ell$-dimensional Lebesgue measure on $\cV$}, denoted $\lambda_\cV$, as follows. Let $\{\bu_1,\dots,\bu_\ell\}$ be an orthonormal basis for~$\cV$, and let $\Phi \colon \R^\ell \to \cV$ be defined by $\Phi(x_1, \ldots, x_\ell) = \sum_{i=1}^\ell x_i \bu_i$. Then we define $\lambda_\cV$ to be the pushforward (see Definition~\ref{pushforward def}) of Lebesgue measure on $\R^\ell$ (which we denote by $\lambda_{\R^\ell}$).
The fact that $\lambda_\cV$ is independent of the choice of orthonormal basis is a consequence of the fact that $\lambda_{\R^\ell}$ is invariant under rigid motions of $\R^\ell$. It is also the case that $\lambda_\cV$ is translation-invariant inside $\cV$, and so it is the same as Haar measure on~$\cV$.

We note that every proper subspace of $\cV$ has $\lambda_\cV$-measure $0$, since every proper subspace of $\R^\ell$ has measure $0$ under the usual Lebesgue measure. We also note that the map $\Phi$ defined above preserves inner products, that is, $\Phi(\bv_1)\cdot \Phi(\bv_2) = \bv_1\cdot\bv_2$ (where the first dot denotes the standard inner product in $\R^\ell$, while the second dot denotes the standard inner product in $\R^n$). Finally, we note that the change of variables formula~\eqref{pushforward change of variables} becomes, in this case,
\begin{equation}
  \label{changeofvar}
  \int_{\cV} g(\bt) \,d\lambda_\cV(\bt) = \int_{\R^\ell} g(\Phi(\bx)) \,d\lambda_{\R^\ell}(\bx) = \int_{-\infty}^\infty \cdots \int_{-\infty}^\infty g \bigg(\sum_{j=1}^{\ell} x_j \bu_j \bigg) \,dx_1 \cdots dx_\ell.
\end{equation}
\end{definition}

We now use  Lemma~\ref{muhatSbounds} to establish the convergence of 
 $\int_{\cV_k^S} \big| \widehat\mu_{k}^S(\bt) \big| \,d \lambda_{\cV_k^S}(\bt)$.

\begin{lemma} \label{absconvintegral}
Let $k$ be a positive integer, and let $\ell$ be the dimension of $\cV_k^S$. If $k > 2\ell$, then $\int_{\cV_k^S} \big| \widehat\mu_{k}^S(\bt) \big| \,d \lambda_{\cV_k^S}(\bt)$ converges.
\end{lemma}

\begin{proof}
As in Definition~\ref{lambda A definition}, let $\{\bu_1,\dots,\bu_\ell\}$ be an orthonormal basis for~$\cV_k^S$, and let $\Phi \colon \R^\ell \to \cV_k^S$ be defined by $\Phi(x_1, \ldots, x_\ell) = \sum_{i=1}^\ell x_i \bu_i$, so that
\[
  \int_{\cV_k^S} \big| \widehat\mu_{k}^S(\bt) \big| \,d\lambda_\cV(\bt) = \int_{\R^\ell} \big| \widehat\mu_{k}^S(\Phi(\bx)) \big| \,d\lambda_{\R^\ell}(\bx)
\]
by equation~\eqref{changeofvar}. If we define the set $\cK$ as in equation~\eqref{compact K definition}, it follows that
\begin{align}
  \int_{\cV_k^S} \big| \widehat\mu_{k}^S(\bt) \big| \,d\lambda_\cV(\bt) &= \int_{\cK} \big| \widehat\mu_{k}^S(\bt) \big| \,d\lambda_\cV(\bt)  + \int_{\cV_k^S \setminus \cK} \big| \widehat\mu_{k}^S(\bt) \big| \,d\lambda_\cV(\bt) \notag \\
  &= \int_{\Phi^{-1}(\cK)} \big|\widehat\mu_{k}^S ( \Phi(\bx) ) \big| \,d\lambda_{\R^\ell}(\bx) + \int_{\R^{\ell} \setminus \Phi^{-1}(\cK)} \big|\widehat\mu_{k}^S (\Phi(\bx)) \big| \,d\lambda_{\R^\ell}(\bx).  \label{I1 and I2}
\end{align}
The set  $\cK$ is compact by Lemma~\ref{muhatSbounds}(b), and hence $\Phi^{-1}(\cK)$ is compact since $\Phi$ is a homeomorphism. Thus the first integral on the right-hand side of equation~\eqref{I1 and I2} is finite, since the integrand is continuous. It therefore suffices to show that the latter integral in equation~\eqref{I1 and I2} is finite.

By Lemma~\ref{muhatSbounds}(a),
\[
   \int_{\R^{\ell} \setminus \Phi^{-1}(\cK)} \big|\widehat\mu_{k}^S (\Phi(\bx)) \big| \,d\lambda_{\R^\ell}(\bx) \ll \int_{\R^{\ell} \setminus \Phi^{-1}(\cK)} \|\Phi(\bx) \|^{-k/2} \,d\lambda_{\R^\ell}(\bx) = \int_{\R^{\ell} \setminus \Phi^{-1}(\cK)} \|\bx\|^{-k/2} \,d\lambda_{\R^\ell}(\bx),
\]
since $\Phi$ preserves inner products. By Lemma~\ref{muhatSbounds}(b), $\cK$ contains a neighborhood of $\bzero$, and hence so does $\Phi^{-1}(\cK)$ since $\Phi$ is continuous; therefore there exists $\ep>0$ such that
\[
  \int_{\R^{\ell} \setminus \Phi^{-1}(\cK)} \big|\widehat\mu_{k}^S (\Phi(\bx)) \big| \,d\lambda_{\R^\ell}(\bx) \ll \int_{\R^{\ell} \setminus \cB_\ep(\bzero)} \|\bx\|^{-k/2} \,d\lambda_{\R^\ell}(\bx).
\]
This integral converges under the assumption $k>2\ell$, as we show by a standard argument using a dyadic version of this integral: for any $R>0$,
\begin{equation*}
    \int_{\cB_{2R}(\bzero) \setminus \cB_R(\bzero)} \|\bx\|^{-k/2} \,d\lambda_{\R^\ell}(\bx) \ll R^{-k/2} \int_{\cB_{2R}(\bzero)} \,d\lambda_{\R^\ell}(\bx) = R^{-k/2} \lambda_{\R^\ell}\big( \cB_{2R}(\bzero) \big) \ll R^{-k/2+\ell}. 
\end{equation*}
It follows that 
\begin{align*}
  \int_{\R^{\ell} \setminus \cB_\ep(\bzero)} \|\bx\|^{-k/2} \,d\lambda_{\R^\ell}(\bx) &= \sum_{j=0}^\infty \int_{\cB_{2^{j+1}\ep}(\bzero) \setminus \cB_{2^j\ep}(\bzero)} \|\bx\|^{-k/2} \,d\lambda_{\R^\ell}(\bx) \\
  &\ll  \sum_{j=0}^{\infty} (2^{j+1}\ep)^{-k/2+\ell} = (2\ep)^{-k/2+\ell}  \sum_{j=0}^{\infty} \big( 2^{-k/2+\ell} \big)^j \ll 1,
\end{align*}
since $-k/2+\ell<0$.
\end{proof}
We now derive a sufficient condition for the absolute continuity of a measure supported on a subspace~$\cV$; this result is standard for measures on $\R^\ell$, and it is not difficult to translate the result to the case of a proper subspace.

\begin{prop}
  \label{abscontV}
Let $\cV$ be a subspace of $\R^r$ with associated Lebesgue measure $\lambda_{\cV}$, and let $\mu$ be a probability measure on $\R^r$.
If $\mu$ is supported on $\cV$ and $\int_{\cV} |\widehat\mu(\bt)| \,d\lambda_{\cV}(\bt)$ converges,
then $\mu$ is absolutely continuous with respect to $\lambda_{\cV}$.  
\end{prop}

\begin{proof}
As in Definition~\ref{lambda A definition}, let $\{\bu_1,\dots,\bu_\ell\}$ be an orthonormal basis for~$\cV$, and let $\Phi \colon \R^\ell \to \cV$ be defined by $\Phi(x_1, \ldots, x_\ell) = \sum_{i=1}^\ell x_i \bu_i$, so that $\Phi^{-1}$ is a homeomorphism from $\cV$ to~$\R^\ell$ (namely, the coordinate map). Define $\nu$ to be the pushforward of $\mu$ to $\R^\ell$ under~$\Phi^{-1}$.

Now for $\bx\in\R^r$, since $\mu$ is supported on $\cV$,
\begin{align*}
\widehat\mu(\Phi(\bx)) = \int_{\R^r} e^{i\Phi(\bx)\cdot\by} \,d\mu(\by) &= \int_\cV e^{i\Phi(\bx)\cdot\by} \,d\mu(\by).
\end{align*}
Therefore, by the change of variables formula~\eqref{pushforward change of variables},
\[
\widehat\mu(\Phi(\bx)) = \int_{\R^\ell} e^{i\Phi(\bx) \cdot \Phi(\by)} \,d\nu(\by) = \int_{\R^\ell} e^{i\bx \cdot \by} \,d\nu(\by) = \widehat\nu(\bx),
\]
since $\Phi$ preserves inner products.
From this we conclude from equation~\eqref{changeofvar} that
\[
\int_{\R^\ell} |\widehat\nu(\bx)| \,d\lambda_{\R^\ell}(\bx) = \int_{\R^\ell} |\widehat\mu(\Phi(\bx))| \,d\lambda_{\R^\ell}(\bx) = \int_{\cV} |\widehat\mu(\bt)| \,d \lambda_{\cV}(\bt) < \infty,
\]
and therefore $\nu$ is absolutely continuous with respect to $\lambda_{\R^\ell}$ by Lemma~\ref{density}(b). Since $\mu$ and $\lambda_\cV$ are the pushforwards of $\nu$ and $\lambda_{\R^r}$, respectively, it follows immediately that $\mu$ is absolutely continuous with respect to $\lambda_\cV$.
\end{proof}

Since $\mu_{k}^S$ is supported on $\cV_{k}^S$, we conclude from Lemma~\ref{absconvintegral} and Proposition~\ref{abscontV}:

\begin{cor}  \label{abscontcor}
If $k>2r$ is an integer, then $\mu_{k}^S$ is absolutely continuous with respect to $\lambda_{\cV_{k}^S}$.
\end{cor}

We remark that Devin~\cite{D} establishes several regularity results (of which absolute continuity is one example) for limiting distributions of explicit formulas derived from any functions from a general class of ``analytic $L$-functions'' analogous to the Selberg class, also for the purpose of analyzing races between counting functions relevant to those $L$-functions.

\section{Deducing that logarithmic densities exist}  \label{log densities section}

We begin this section by deriving, from the fact that $\mu_{k}^S$ is absolutely continuous with respect to $\lambda_{\cV_{k}^S}$, the conclusion that equation~\eqref{mu defined here} holds not just for continuous functions but also for the indicator function of the wedge $\cS$ defined in equation~\eqref{wedge definition}; in particular, this will establish Theorem~\ref{inclusive theorem}(a). Part of this deduction requires showing that the hyperplanes bounding this wedge are not assigned mass by the measure $\mu$ defined in equation~\eqref{mu defined here}; the following lemma suffices for this purpose.

\begin{lemma}  \label{conv vanish subsp}
Let $\cV$ and $\cW$ be subspaces of $\R^r$, and assume that $\cV$ is not contained in $\cW$. 
Let $\mu_1$ be a measure supported on $\cV$ that is absolutely continuous with respect to $\lambda_\cV$,
Lebesgue measure on $\cV$, and let $\mu_2$ be any measure on $\R^r$.  Then $(\mu_1*\mu_2)(\cW) = 0$.
\end{lemma}
\begin{proof}
Let $\cY = \cV \cap \cW$.  Since $\cV$ is not contained in $\cW$, we see that $\cY$ is a proper subspace of $\cV$. 
By definition and the Fubini--Tonelli theorem,
\[
   (\mu_1*\mu_2)(\cW) = \int_{\bx \in \R^r} \bigg( \int_{\by \in \R^r} 1_\cW(\bx+\by) \,d\mu_1(\by) \bigg) \,d\mu_2(\bx) .
\]
Since $\mu_1$ is supported on $\cV$, for any fixed $\bx$, we have
\begin{equation}
     \int_{\by \in \R^r} 1_\cW(\bx+\by) \,d\mu_1(\by) = \int_{\by \in \R^r} 1_\cY(\bx+\by) \,d\mu_1(\by) = \int_{\by \in \R^r } 1_{\cY-\bx}(\by) \,d\mu_1(\by) = \mu_1(\cY-\bx).
\end{equation}
Note that if $\bx \notin \cV$, then $(\cY-\bx) \cap \cV = \emptyset$ and thus $\mu_1(\cY-\bx)=0$ as $\mu_1$ is supported on $\cV$.  
Since $\cY$ is a proper subspace of $\cV$, it follows that $\lambda_\cV(\cY)=0$.  Furthermore, $\lambda_\cV(\cY-\bx)=0$ for any $\bx \in \cV$. 
Therefore since $\mu_1$ is absolutely continuous with respect to $\lambda_\cV$, it follows that $\mu_1(\cY-\bx)=0$  
for all $\bx \in \R^r$. We conclude that
\begin{equation}
\begin{split}
   (\mu_1*\mu_2)(\cW) 
& = \int_{\bx \in \R^r} 
  \mu_1(\cY-\bx)  \,d\mu_2(\bx) = \int_{\bx \in \R^r} 0  \, d\mu_2(\bx) = 0
\end{split}
\end{equation} 
as desired.
\end{proof}

We will require a version of this last lemma in which the subspace~$\cW$ is replaced by a sphere (the boundary of a ball). A nearly identical proof, using as its key input the fact that spheres have Lebesgue measure~$0$ in any subspace of positive dimension, yields the following result.

\begin{lemma}  \label{conv vanish sphere}
Let $\cV\ne\{\bzero\}$ be a subspace of $\R^r$, and let $\cW$ be a sphere in $\R^r$.
Let $\mu_1$ be a measure supported on $\cV$ that is absolutely continuous with respect to $\lambda_\cV$,
Lebesgue measure on $\cV$, and let $\mu_2$ be any measure on $\R^r$.  Then $(\mu_1*\mu_2)(\cW) = 0$.
\end{lemma}

\begin{proof}[Proof of Theorem~\ref{inclusive theorem}(a)]
Recall that $\cS$ is the wedge defined in equation~\eqref{wedge definition}. Since
\[
   \mu(\cS) = \int_{\R^r} \one_{\cS}(\bx) \, d \mu(\bx),
\]
we will construct bounded continuous majorants and minorants for the function~$\one_{\cS}$.
Given $\varepsilon >0$, let $g_{\varepsilon}(x)$ be a bounded continuous minorant of $\one_{[0,\infty)}$ such that $g_{\varepsilon}(x)=0$
for $x \le 0$, $g_{\varepsilon}(x)=1$ for $x \in [\varepsilon, \infty)$, and $0 < g_{\varepsilon}(x) < 1$ for $x \in (0,\varepsilon)$. 
It follows that 
\[
   h_{\varepsilon}^{-}(\bx) := g_{\varepsilon}( \min( x_1-x_2, x_2-x_3, \ldots, x_{r-1}-x_r)) 
\]
and 
\[
   h_{\varepsilon}^{+}(\bx) := g_{\varepsilon}(\varepsilon+ \min( x_1-x_2, x_2-x_3, \ldots, x_{r-1}-x_r)) 
\]
are a bounded continuous minorant and bounded continuous majorant, respectively, of $\one_{\cS}(\bx)$. 
Consequently, equation~\eqref{mu defined here}
implies
\begin{equation}
  \label{limdish}
  \lim_{y \to \infty} \frac1y  \int_0^y h_{\varepsilon}^{\pm} (E(e^t)) \,dt 
  = \int_{\R^r} h_{\varepsilon}^{\pm}(\bx) \,d \mu(\bx)
\end{equation}
(where $E$ was defined in equation~\eqref{Ey}).
Also, for $y >0$,
\begin{equation}
    \frac1y  \int_0^y h_{\varepsilon}^{-} (E(e^t)) \,dt 
    \le 
     \frac1y  \int_0^y \one_{\cS}(E(e^t)) \,dt 
    \le   \frac1y  \int_0^y h_{\varepsilon}^{+} (E(e^t)) \,dt. 
\end{equation}
Letting $y \to \infty$ and using equation~\eqref{limdish}, 
\begin{equation}
    \int_{\R^r} h_{\varepsilon}^{-}(\bx) \,d \mu(\bx) 
    \le  \liminf_{y\to\infty} \frac1y  \int_0^y \one_{\cS}(E(e^t)) \,dt 
    \le \limsup_{y\to\infty} \frac1y  \int_0^y \one_{\cS}(E(e^t)) \,dt \le
     \int_{\R^r} h_{\varepsilon}^{+}(\bx) \,d \mu(\bx).  \label{sup and inf separate for now}
\end{equation}
Thus
\begin{multline}
  0 \le   \limsup_{y\to\infty} \frac1y  \int_0^y \one_{\cS}(E(e^t)) \,dt - \liminf_{y\to\infty} \frac1y  \int_0^y \one_{\cS}(E(e^t)) \,dt \\
  \le   \int_{\R^r} (h_{\varepsilon}^{+}(\bx)- h_{\varepsilon}^{-}(\bx)) \,d \mu(\bx)
   \le \int_{\cX_{\varepsilon}} \,d \mu(\bx)  \label{limsup and liminf are close}
\end{multline}
where 
$\cX_{\varepsilon} =  \bigcup_{j=1}^{r-1} \{ x \in \R^r \colon \text{dist}(x,\cW_j) < \varepsilon \}$
and $\cW_j = \{ x \in \R^r \colon x_j = x_{j+1} \}$.
It follows from the dominated convergence theorem that 
\[
  \lim_{\varepsilon \to 0+}  \int_{\cX_{\varepsilon}} \,d \mu(\bx) = \int_{\cX} \,d \mu(\bx) = \mu(\cX)
\]
where $\cX = \bigcup_{j=1}^{r-1} \cW_j$. Therefore taking the limit of equation~\eqref{limsup and liminf are close} as ${\varepsilon \to 0+}$ yields
\begin{equation}
  0 \le   \limsup_{y\to\infty} \frac1y  \int_0^y \one_{\cS}(E(e^t)) \,dt - \liminf_{y\to\infty} \frac1y  \int_0^y \one_{\cS}(E(e^t)) \,dt
  \le \int_{\cX} \,d \mu(\bx) = \mu(\cX).  \label{does X have measure 0}
\end{equation}

Set $k=2r+1$.  By the assumption of Theorem~\ref{inclusive theorem}(a) there is at least one character $\chi$ 
which is $k$-sturdy.  Thus, by Proposition~\ref{muconvolutionsturdy}, we know that $\mu=\mu_{k}^S*\mu^N_k$, where $\mu_{k}^S$
is supported on $\cV_{k}^S$ as per Definition~\ref{xyz vectors definition}.
By assumption, $\cV^S_k$ is either $\R^r$ or a dimension-$(r-1)$ subspace of $\R^r$ not containing $(1,\ldots, 1)$.
Notice that $\cV^S_k$ is not contained in any boundary hyperplane $\cW_j$: this is obvious if $\cV^S_k=\R^r$, while otherwise $\cV^S_k$ is a dimension-$(r-1)$ subspace of $\R^r$ not containing $(1,\ldots, 1)$, whereas $\cW_j$ is a dimension-$(r-1)$ subspace of 
$\R^r$ containing $(1,\dots,1)$. Therefore we may apply Lemma~\ref{conv vanish subsp} with $\cV=\cV^S_k$ and $\cW=\cW_j$, and $\mu_1=\mu^S_k$ (which we know is absolutely continuous with respect to $\lambda_\cV$ by Corollary~\ref{abscontcor}) and $\mu_2=\mu^N_k$, to see that every $\mu(\cW_j)$ equals~$0$. Consequently $\mu(\cX)=0$ as well, and we conclude from equation~\eqref{does X have measure 0} that
$\lim_{y \to \infty} \frac1y \int_0^y \one_{\cS}(E(e^t)) \,dt$ exists.

Thus we may rewrite equation~\eqref{sup and inf separate for now} as
\begin{equation}
   \int_{\R^r} h_{\varepsilon}^{-}(\bx) \,d \mu(\bx) 
    \le  \lim_{y \to \infty} \frac1y  \int_0^y \one_{\cS}(E(e^t)) \,dt
     \le
     \int_{\R^r} h_{\varepsilon}^{+}(\bx) \,d \mu(\bx). 
     \label{one limit left to take}
\end{equation}
Again the dominated convergence theorem shows that 
\[
  \lim_{\varepsilon \to 0+}   \int_{\R^r} h_{\varepsilon}^{\pm}(\bx) \,d \mu(\bx) 
  = \int_{\R^r} \one_{\cS}(\bx) \,d \mu(\bx) ,
\]
and hence taking limits in equation~\eqref{one limit left to take} as $\varepsilon\to0+$ yields
\begin{equation} \label{log density equality}
    \lim_{y \to \infty} \frac1y  \int_0^y \one_{\cS}(E(e^t)) \,dt  
     = \int_{\R^r} \one_{\cS}(\bx) \,d \mu(\bx) = \mu(\cS)
\end{equation}
as desired. In particular, in light of equation~\eqref{relevant delta}, the logarithmic density $\delta\big( \{t\ge1\colon E(t)\in\cS\} \big)$ exists (and equals $\mu(\cS)$).
\end{proof}

A careful examination of the proof reveals that we can actually prove a statement that is somewhat stronger than Theorem~\ref{inclusive theorem}(a), but with a more technical hypothesis.

\begin{theorem}
Assume GRH. Let $a_1,\dots,a_r$ be distinct reduced residues\mod q. Suppose that there exists a positive integer $k$ such that
\begin{itemize}
\item $\cV_k^S$ is not contained in any ``diagonal hyperplane'' $\cW_{i,j} = \{ \bx\in\R^r\colon x_i=x_j \}$ with $1\le i<j\le r$; and
\item $k > 2\dim(\cV_k^S)$.
\end{itemize}
Then the prime number race among $a_1,\dots,a_r$\mod q is weakly inclusive.
\end{theorem}

Note that Theorem~\ref{inclusive theorem}(a) implies the following for two-way races: if there exists a single character $\chi\mod q$ (satisfying $\chi(a)\ne\chi(b)$) with five self-sufficient zeros, then the logarithmic density of the race between $\pi(x;q,a)$ and $\pi(x;q,b)$ exists. Theorem~\ref{two-way theorem}(a) improves upon this in two ways: it lowers the number of self-sufficient zeros required from five to three, and it allows those self-sufficient zeros to belong to different characters $\chi\mod q$ (which will, by definition, be $1$-sturdy at the very least).

In this situation, the definition~\eqref{muSdefn} becomes the $\R^2$-valued random vector
\begin{equation*}
X_1^S
=  2 \Re \sum_{\substack{\chi\mod q \\ \chi\ne\chi_0 \\ \chi \text{ $1$-sturdy}}} \big( \overline{\chi}(a), \overline{\chi}(b) \big) \sum_{\gamma \in \Gamma^S(\chi)} \frac{Z_\gamma}{\sqrt{\frac14+\gamma^2}}
\end{equation*}
whose distribution $\mu_1^S$, as per Proposition~\ref{muconvolutionsturdy}, has characteristic function
\begin{equation*}
  \widehat\mu_1^S(\bt)  
  = \prod_{\substack{\chi\mod q \\ \chi\ne\chi_0 \\ \chi \text{ $1$-sturdy}}}  \prod_{\gamma \in \Gamma^S(\chi)} J_0 \Bigg( 
  \frac{ 2 |  \chi(a)t_1 + \chi(b)t_2 | }{ \sqrt{\frac{1}{4}+\gamma^2}}
  \Bigg). 
\end{equation*}
However, since we are mainly interested in the difference $\pi(x;q,a)-\pi(x;q,b)$ rather than the absolute sizes of $\pi(x;q,a)$ and $\pi(x;q,b)$ separately, we can modify this approach by postcomposing with the function $(t_1,t_2)\mapsto t_1-t_2$ from $\R^2$ to~$\R$, thereby considering the $\R$-valued random variable
\begin{equation*}
\underline X_1^S
=  2 \Re \sum_{\substack{\chi\mod q \\ \chi\ne\chi_0 \\ \chi \text{ $1$-sturdy}}} \big( \overline{\chi}(a) - \overline{\chi}(b) \big) \sum_{\gamma \in \Gamma^S(\chi)} \frac{Z_\gamma}{\sqrt{\frac14+\gamma^2}}
\end{equation*}
whose distribution $\underline\mu_1^S$ has characteristic function
\begin{equation}  \label{underline mu one hat}
  \widehat{\underline\mu}_1^S(t)  
  = \prod_{\substack{\chi\mod q \\ \chi\ne\chi_0 \\ \chi \text{ $1$-sturdy}}}  \prod_{\gamma \in \Gamma^S(\chi)} J_0 \Bigg( 
  \frac{ 2 \big| \big( \chi(a)-\chi(b) \big)t \big| }{ \sqrt{\frac{1}{4}+\gamma^2}}
  \Bigg). 
\end{equation}
(As a reality check, notice that under the assumption of LI, so that all characters are $k$-sturdy and all zeros are represented in $\Gamma^S(\chi)$, this formula is the same as the one derived by Fiorilli and the first author---see \cite[Definitions 2.4 and 2.11 and Propositions 2.6 and~2.13]{FM}.)

\begin{proof}[Proof of Theorem~\ref{two-way theorem}(a)]
By the inequality~\eqref{J0bound}, each factor in the product in equation~\eqref{underline mu one hat} is at most $1$ in absolute value; so if we retain only the factors corresponding to the three hypothesized self-sufficient zeros, we obtain (by analogy with equation~\eqref{applied Bessel bound})
\[
  | \widehat{\underline\mu}_1^S(t)  |
  \ll \min\bigg\{ 1, \frac1{|t|^{3/2}} \bigg\}
\]
where the implicit constant depends on the specific self-sufficient zeros and associated character values (note here that it is crucial that $|\chi(a)-\chi(b)|\ne0$). This upper bound is sufficient to show that $\widehat{\underline\mu}_1^S(\xi)$ is absolutely integrable; hence by Lemma~\ref{density}, $\underline\mu_1^S$ is absolutely continuous with respect to Lebesgue measure on $\R$, and so (by comparison with the proof of Theorem~\ref{inclusive theorem}(a)),
\[
     \lim_{X \to \infty} \frac{1}{\log X}  \int_1^{X} 
     \one_{(0,\infty)} \Big( \frac{\log x}{\sqrt{x}} (E(t;q,a)-E(t;q,b))
     \Big) \, \frac{dx}{x}  
     = \int_0^\infty d \underline\mu_1(x) ,
\]
where $E(t;q,a)$ was defined in equation~\eqref{Exqa}.
In particular, in light of equation~\eqref{relevant delta}, the logarithmic density $\delta\big( \{t\ge1\colon E(t;q,a)>E(t;q,b)\} \big)$ exists (and equals $\underline\mu_1\big( (0,\infty) \big)$).
\end{proof}

We may consider the entire set of primes (the ``$q=1$ case''), whose counting function is of course $\pi(x)$. Since there are no longer multiple residue classes, however, we must create a second contestant; using the logarithmic integral $\li(x)$ as this second contestant recovers the classical question of investigating when $\pi(x)>\li(x)$.

\begin{proof}[Proof of Theorem~\ref{pili theorem}(a)]
As the proof is nearly identical to the proof of Theorem~\ref{two-way theorem}(a), we only indicate the differences briefly.  
Let $\mu_\pi$ denote the limiting logarithmic density of $\frac{\log x}{\sqrt{x}} (\pi(x)-\text{li}(x))$. Using the explicit formula
\begin{equation*}
  \frac{\log x}{\sqrt{x}} \big( \pi(x) - \li(x) \big) = -1+\sum_{\substack{\gamma \in \mathbb{R} \\ \zeta(\frac{1}{2}+i \gamma,\chi)=0} } \frac{x^{i \gamma}}{\frac{1}{2}+i \gamma} + O \bigg( \frac{1}{\log x} \bigg) = -1+2 \Re \sum_{\gamma\in\Gamma(\chi_0)} \frac{x^{i \gamma}}{\frac{1}{2}+i \gamma} + O \bigg( \frac{1}{\log x} \bigg),
\end{equation*}
we can show, analogously to Proposition~\ref{existence of muN prop}, that
$\mu_\pi =\mu_\pi^{R} * \mu_\pi^{N}$ where $\mu_\pi^R$ is the probability measure associated to the random variable
\begin{equation}  \label{Xpi definition}
  X_\pi^R = 2 \Re  \sum_{\substack{\gamma \in \Gamma^S(\chi_0) }} \frac{Z_\gamma}{\sqrt{\frac14+\gamma^2}},
\end{equation}
whose characteristic function is
\[
    \widehat{\mu}_\pi^R(t) = \prod_{\gamma \in \Gamma^S(\chi_0)} J_0 \Bigg( 
  \frac{ 2 t }{ \sqrt{\frac{1}{4}+\gamma^2}}
  \Bigg). 
\]
If there are three self-sufficient zeros then, analogously to Lemma~\ref{absconvintegral} and its proof, 
$| \widehat{\mu}_\pi^R(t)| \ll \min\{ 1, |t|^{-3/2} \}$.   This decay rate ensures that 
$\widehat{\mu}_\pi^R(t)$ is absolutely integrable, so that $\mu_\pi^R$ is absolutely continuous with respect to Lebesgue measure on $\R$ by Lemma~\ref{density}. Hence by the method of proof of Theorem~\ref{inclusive theorem}(a), we can conclude that
\[
     \lim_{X \to \infty} \frac{1}{\log X}  \int_1^X \one_{(0,\infty)} \bigg(\frac{\log x}{\sqrt{x}} 
     (\pi(x)-\li(x)) \bigg) \, \frac{dx}{x}
     = \int_0^\infty d \mu_\pi (x).
\]
In particular, in light of equation~\eqref{relevant delta}, the logarithmic density $\delta\big( \{t\ge1\colon \pi(x) > \li(x) \} \big)$ exists (and equals $\mu_\pi\big( (0,\infty) \big)$).
\end{proof}

\section{$\mu_W^R$ is present in cylinders and balls}  \label{cylinder section}

The goal of this section is to prove that various prime number races are inclusive or strongly inclusive (in particular, to establish Theorem~\ref{inclusive theorem}(b)--(c)), by showing the positivity of various logarithmic densities or, equivalently, of the measures of various sets under probability measures such as~$\mu$. The key to this goal is to show, under suitable assumptions, that
$\mu^R_W$ assigns strictly positive measure to every infinite cylinder parallel to $(1,1, \ldots, 1)$ (Proposition~\ref{omnipresent}) or to every ball (Proposition~\ref{omnipresent, strong}).
Our method of proof is to convert this question to one about specific random variables taking values in $\R^r$, which the first two lemmas of this section will help us understand. After giving the proof of Theorem~\ref{inclusive theorem}(b) (and showing that its spanning hypothesis is as general as could be hoped) and Theorem~\ref{inclusive theorem}(c), we also prove Theorem~\ref{pili theorem}(b)--(c) and state two variants for primes in arithmetic progressions (Theorem~\ref{pili theorem for APs}).

We begin with three elementary lemmas.

\begin{lemma}
Let $\lambda_1, \dots, \lambda_N$ be positive real numbers such that $\sum_{n=1}^N \lambda_n > 2\max\{\lambda_1,\dots,\lambda_N\}$. Then for any complex number $z$ with $|z| < \sum_{n=1}^N \lambda_n$, there exist unimodular numbers $e^{i\theta_1}, \dots e^{i\theta_N}$ such that $\sum_{n=1}^N \lambda_n e^{i\theta_n} = z$.
\label{robot arm lemma}
\end{lemma}

\noindent Since the hypotheses imply that no element of $\{ \lambda_1, \dots, \lambda_N, |z| \}$ exceeds the sum of the other elements, the lemma follows easily from the existence of a convex polygon in the plane with side lengths $\{ \lambda_1, \dots, \lambda_N, |z| \}$.


\begin{lemma}  \label{robot variable lemma}
Let $\Gamma$ be a finite set of real numbers, and let $\{ Z_\gamma\colon \gamma\in\Gamma\}$ be a collection of independent random variables each of which is uniformly distributed on the unit circle in~$\C$. Let $z$ be a complex number, and let $\{\lambda_\gamma\colon \gamma\in\Gamma\}$ be positive real numbers satisfying $\max \{ \lambda_\gamma \colon \gamma\in\Gamma \} \le 4$ and $\sum_{\gamma\in\Gamma} \lambda_\gamma > \max\{8,|z|\}$. Then for any $\varepsilon>0$, there is a positive probability that
\begin{equation}  \label{lambdagamma ep}
\bigg |\sum_{\gamma\in\Gamma} \lambda_\gamma Z_\gamma - z \bigg| < \varepsilon.
\end{equation}
\end{lemma}

\begin{proof}
Set $N=\#\Gamma$, and consider the map $f\colon \T^N\to\C$ defined by $f(\{\theta_\gamma\colon \gamma\in\Gamma\}) = \sum_{\gamma\in\Gamma} \lambda_\gamma e^{i\theta_\gamma}$. The hypotheses $\max \{ \lambda_\gamma \colon \gamma\in\Gamma \} \le 4$ and $\sum_{\gamma\in\Gamma} \lambda_\gamma > \max\{8,|z|\}$ allow us to apply Lemma~\ref{robot arm lemma} and conclude that $z$ is in the image of~$f$. By the continuity of $f$, there exists an open subset $\cU$ of $\T^N$ on which the value of $f$ is within $\varepsilon$ of~$z$. On the other hand, the probability distribution of $\sum_{\gamma\in\Gamma} \lambda_\gamma Z_\gamma$ is the same as the pushforward of Haar measure on the torus $\T^N$ under the function~$f$. In particular, the probability that the inequality~\eqref{lambdagamma ep} holds is at least the measure of $\cU$, and the measure of any open subset of $\T^N$ is positive.
\end{proof}

We use $\|{\cdot}\|_\infty$ to denote the maximum absolute value of the entries of a vector or matrix.

\begin{lemma} \label{matrixlemma}
Let $\bt \in \R^r$ and let $\bA$ be an $r \times m$ real matrix. If  $\bt$ is in the column space of~$\bA$,
then there exists $\by \in \R^m$ such that $\bA\by=\bt$ and 
\[
   \| \by \|_{\infty} \ll_{\bA} \| \bt \|_{\infty}. 
\]
\end{lemma}

\begin{proof} 
Let $\bE$ denote the reduced row echelon form of $\bA$, which has pivots in the first~$k$ rows where $k=\mathop{\rm rank}\bA$. There exists an invertible $r\times r$ matrix~$\bU$ such that $\bE =\bU\bA$.
Because~$\bt$ is in the column space of~$\bA$ by assumption, we know that~$\bU\bt$ is in the column space of~$\bE$; in particular, we can write the column vector~$\bU\bt$ as $(u_1\ u_2\ \cdots\ u_k\ 0\ \cdots\ 0)^T$.
Define a column vector $\by = (y_1\ y_2\ \cdots\ y_m)^T$ by setting $y_j=u_i$ if~$\bE$ has a pivot in its $i$th row and $j$th column, and $y_j=0$ otherwise. By construction, $\bE\by=\bU\bt$ exactly, and so $\bA\by = \bU^{-1}\bE\by=\bU^{-1}\bU\bt = \bt$. Furthermore,
\[
   \| \by \|_{\infty} \le \| \bU \bt \|_\infty  \le r \| \bU \|_\infty \| \bt \|_{\infty}
   \ll_{\bA}  \| \bt \|_{\infty}
\]
as desired. 
\end{proof}

We are now ready to establish sufficient conditions for the probability measure $\mu^R_W$ defined in Proposition~\ref{muconvolutionrobust} to assign positive mass to cylinders and balls in~$\R^r$.

\begin{definition} \label{cylinder def}
Let $\cC_\rho(\bx)$ denote the (infinite, open) cylinder in $\mathbb{R}^r$ with radius $\rho$, center $\bx$, and axis parallel to $(1,\dots,1)$, defined as
\begin{equation*}
\cC_\rho(\bx) = \bigcup_{u\in\R} \cB_\rho\big( \bx+u(1,\dots,1) \big)
\end{equation*}
(using the notation for open balls from Definition~\ref{ball def}).
Of course, the ``center'' of a cylinder is not uniquely defined: any point $\bx+u(1,\dots,1)$ on the central axis could be used.  
\end{definition}

\begin{prop}  \label{omnipresent temp try}
For every nonempty set of nonprincipal characters $\cQ\mod q$ and every real number $V>0$, there exists a constant $W(\cQ,V)$ with the following property:
Suppose that $\bt \in \R^{r}$ satisfies $\|\bt\| \le V$ and is in the span of the set of vectors
\begin{equation} \label{set of vectors}
\big\{ (1,\dots,1) \big\} \cup \big\{ \bx_\chi \colon \chi\in\cQ \big\} \cup \big\{ \by_\chi \colon \chi\in\cQ \big\},
\end{equation}
with $\bx_\chi$ and $\by_\chi$ from Definition~\ref{ETx def}.
If $W \ge W(\cQ,V)$ is a real number such that every $\chi\in\cQ$ is $W$-robust, then $\mu^R_W(\cC_\rho(\bt)) > 0$ for every $\rho >0$.
\end{prop}

\begin{proof}
By assumption,~$\bt^T$ is in the column space of the matrix~$\bA=\bA(\cQ)$ whose columns are the transposes of the vectors in equation~\eqref{set of vectors}. By Lemma~\ref{matrixlemma}, there exists a positive constant~$c_\bA$ and real numbers $u$ and $\{ v_\chi \colon \chi\in\cQ \}$ and $\{ w_\chi \colon \chi\in\cQ \}$ such that
\begin{equation} \label{spanning hypothesis}
\bt = u (1,\dots,1) + \sum_{\chi\in\cQ} ( 2v_\chi \bx_\chi + 2w_\chi \by_\chi )
\end{equation}
(where the factors of~$2$ are for later convenience) and
\begin{equation} \label{vchiwchibd}
\max\{|v_\chi|, |w_{\chi}|\} \le c_\bA V \quad\text{for each }\chi\in\cQ.
\end{equation}
We now choose
\begin{equation}
  \label{Wdefn}
 W(\cQ,V) = \max\{8,c_\bA V\sqrt2\} + \frac14\sum_{\gamma\in\Gamma^S(\chi)} \frac{1}{\gamma^3}
\end{equation}

Assume that $W \ge W(\cQ,V)$ and that every $\chi \in \cQ$ is $W$-robust.
For any $W$-robust nonprincipal character~$\chi$ that is not in~$\cQ$, we define $v_\chi=w_\chi=0$, so that we have now written
\begin{equation*}
\bt = u (1,\dots,1) + \sum_{\substack{\chi\mod q \\ \chi\ne\chi_0 \\ \chi \text{ $W$-robust}}} ( 2v_\chi \bx_\chi + 2w_\chi \by_\chi ).
\end{equation*}
By definition, for every $W$-robust character~$\chi$, when~$T$ is sufficiently large we have
 \begin{equation}
   \label{zerosumtruncated}
 \sum_{\substack{\gamma\in\Gamma^S(\chi) \\ \gamma \le T} }
   \frac{1}{\gamma} \ge \frac{W}{2}.
\end{equation}
The elementary inequality ${2}/{\sqrt{1/4+x^2}} > {2/x} - 1/4x^3$ for $x>0$ implies that for any $T>0$,
\begin{equation} \label{zerosuminequality}
   \sum_{\substack{\gamma\in\Gamma^S(\chi) \\ \gamma\le T}} \frac2{\sqrt{\frac14+\gamma^2}} 
   > 2   \sum_{\substack{\gamma\in\Gamma^S(\chi) \\ \gamma\le T}} \frac{1}{\gamma}-
   \frac{1}{4}
   \sum_{\gamma\in\Gamma^S(\chi) }
   \frac{1}{\gamma^3}.
\end{equation}
from which it follows, using equations~\eqref{zerosumtruncated}, \eqref{Wdefn}, and~\eqref{vchiwchibd}, that when~$T$ is sufficiently large,
\begin{align}
\sum_{\substack{\gamma\in\Gamma^S(\chi) \\ \gamma\le T}} \frac2{\sqrt{\frac14+\gamma^2}} &\ge W - \frac14 \sum_{\gamma\in\Gamma^S(\chi)} \frac{1}{\gamma^3} \notag \\
&\ge W(\cQ,V) - \frac14 \sum_{\gamma\in\Gamma^S(\chi)} \frac{1}{\gamma^3} \notag \\
&= \max\{8,c_\bA V\sqrt2\} \ge \max\big\{8,\max\big\{ \sqrt{v_\chi^2+w_\chi^2}\colon \chi\in\cQ \big\} \big\}.
\label{inequality2}
\end{align}

Let $\rho>0$ be given; our goal is to show that $\mu^R_W(\cC_\rho(\bt) )$ is positive.
By equation~\eqref{pre yes it converges}, the series $\sum_{\gamma\in\Gamma^S(\chi)} 1/(\frac14+\gamma^2)$
converges for any character~$\chi$, and so when~$T$ is sufficiently large,
\begin{equation} \label{inequality1}
\sum_{\substack{\gamma\in\Gamma^S(\chi) \\ \gamma>T}} \frac1{\frac14+\gamma^2} \le \sum_{\substack{\gamma\in\Gamma^S(\chi) \\ \gamma>T}} \frac1{\frac14+\gamma^2} < \frac{\rho^2}{64 r \phi(q)^2} .
\end{equation}
By Proposition~\ref{muconvolutionrobust},
the probability measure $\mu_W^R$ is the distribution of the random vector $X_W^R$ defined in equation~\eqref{XR def}, which we write as
\begin{equation} \label{Zss1 definition}
X_W^R =  2\Re \sum_{\substack{\chi\mod q \\ \chi \ne \chi_0 \\ \chi \text{ $W$-robust}}} (\bx_\chi - i\by_\chi) \sum_{\gamma\in\Gamma^S(\chi)} \frac{Z_\gamma}{\sqrt{\frac14+\gamma^2}}.
\end{equation}
Therefore we need to show that $P( X^R \in \cC_\rho(\bt))$ is positive.

Fix~$T$ large enough that both inequalities~\eqref{inequality2} and~\eqref{inequality1} are satisfied for every nonprincipal $W$-robust character~$\chi$. For each such character we define random variables $U_\chi^S$ and $V_\chi^S$ by
\[
U_\chi^S = 2 \sum_{\substack{\gamma\in\Gamma^S(\chi) \\ \gamma\le T}} \frac{Z_{\gamma}}{\sqrt{\frac14+\gamma^2}}
\quad\text{and}\quad
V_\chi^S = 2 \sum_{\substack{\gamma\in\Gamma^S(\chi) \\ \gamma > T}} \frac{Z_{\gamma}}{\sqrt{\frac14+\gamma^2}}.
\]
We now claim that
\begin{align}
  \label{PApositive2}
  P( X_W^R \in \cC_\rho(\bt)) &\ge P\big( X_W^R \in \cB_\rho\big(\bt-u(1,\dots,1) \big) \big) \\
  &\ge \prod_{\substack{\chi\mod q \\ \chi\ne\chi_0 \\ \chi\text{ $W$-robust}}} P \bigg( |U_\chi^S - (v_\chi-iw_\chi)| < \frac{\rho}{4 \sqrt{r} \phi(q)} \bigg) \prod_{\substack{\chi\mod q \\ \chi\ne\chi_0 \\ \chi\text{ $W$-robust}}} P \bigg( |V_\chi^S| < \frac{\rho}{4 \sqrt{r} \phi(q)} \bigg), \notag
\end{align}
where the first inequality is obvious from Definition~\ref{cylinder def}. The proposition will be established once we justify this inequality and show that each factor on the right-hand side is strictly positive.

To establish the inequality~\eqref{PApositive2}, we essentially use the inequality~\eqref{convolution lower bound} for the distributions of these random variables, but with more than two summands.
First we relate $X^R_W$ to the random variables~$U_\chi^S$ and~$V_\chi^S$.
From equation~\eqref{Zss1 definition}, 
\begin{align*}
X_W^R &= 2\Re \sum_{\substack{\chi\mod q \\ \chi\ne\chi_0 \\ \chi\text{ $W$-robust}}} (\bx_\chi - i\by_\chi) (U_\chi^S + V_\chi^S) \\
&= 2\sum_{\substack{\chi\mod q \\ \chi\ne\chi_0 \\ \chi\text{ $W$-robust}}} (\bx_\chi \Re U_\chi^S + \by_\chi \Im U_\chi^S) + 2\sum_{\substack{\chi\mod q \\ \chi\ne\chi_0 \\ \chi\text{ $W$-robust}}} (\bx_\chi \Re V_\chi^S + \by_\chi \Im V_\chi^S),
\end{align*}
and thus
\begin{align*}
  X_W^R - & \big( \bt - u(1,\ldots, 1) \big) \\
  & = 2\sum_{\substack{\chi\mod q \\ \chi\ne\chi_0 \\ \chi\text{ $W$-robust}}} \big( \bx_\chi (\Re U_\chi^S - v_\chi) + \by_\chi (\Im U_\chi^S - w_\chi) \big)
  + 2\sum_{\substack{\chi\mod q \\ \chi\ne\chi_0 \\ \chi\text{ $W$-robust}}} (\bx_\chi \Re V_\chi^S + \by_\chi \Im V_\chi^S) \\
  & =   2\sum_{\substack{\chi\mod q \\ \chi\ne\chi_0 \\ \chi\text{ $W$-robust}}}
  \Re \big( (\bx_\chi-i\by_\chi)(U_\chi^S-(v_\chi+iw_\chi) ) \big)
  + 2\sum_{\substack{\chi\mod q \\ \chi\ne\chi_0 \\ \chi\text{ $W$-robust}}} \Re \big((\bx_\chi-i\by_\chi) V_\chi^S  \big).
\end{align*}
By the triangle inequality,
\begin{align}
  \big| X_W^R - \big( \bt - u(1,\ldots, 1) \big) \big| 
  &  \le 2\sum_{\substack{\chi\mod q \\ \chi\ne\chi_0 \\ \chi\text{ $W$-robust}}} |\bx_\chi-i\by_\chi| \big|U_\chi^S-(v_\chi+iw_\chi) \big| + 2\sum_{\substack{\chi\mod q \\ \chi\ne\chi_0 \\ \chi\text{ $W$-robust}}} |\bx_\chi-i\by_\chi| |V_\chi^S|  \notag \\ 
  & \le 2\sqrt r \sum_{\substack{\chi\mod q \\ \chi\ne\chi_0 \\ \chi\text{ $W$-robust}}} \big|U_\chi^S-(v_\chi+iw_\chi) \big| + 2\sqrt r \sum_{\substack{\chi\mod q \\ \chi\ne\chi_0 \\ \chi\text{ $W$-robust}}} |V_\chi^S|,
\label{keyinequality}
\end{align}
since $ |\bx_\chi-i\by_\chi|=|(\overline\chi(a_1),\ldots,\overline\chi(a_r))| = \sqrt{r}$.

Note that 
$\{ U_\chi^S, V_\chi^S \colon \chi \ne \chi_0 \}$
is a set of mutually independent random variables, since all of the~$Z_\gamma$ are mutually independent; so the right-hand side of equation~\eqref{PApositive2} is the probability that each $U_\chi^S$ is within ${\rho}/{4 \sqrt{r} \phi(q)}$ of $v_\chi+iw_\chi$ and, simultaneously, each $V_\chi^S$ is within ${\rho}/{4 \sqrt{r} \phi(q)}$ of~$0$.
In the event where each of the $U_{\chi}^S$ and $V_{\chi}^{S}$ satisfy these inequalities, it follows from equation~\eqref{keyinequality} that 
\begin{align*}
  \big| X_W^R - \big( \bt - u(1,\ldots, 1) \big) \big| 
  &\le 2\sqrt r \sum_{\substack{\chi\mod q \\ \chi\ne\chi_0 \\ \chi\text{ $W$-robust}}} \frac{\rho}{4\sqrt r\phi (q)} + 2\sqrt r \sum_{\substack{\chi\mod q \\ \chi\ne\chi_0 \\ \chi\text{ $W$-robust}}} \frac{\rho}{4\sqrt r\phi (q)} \\
  &= \frac12(\phi(q)-1) \frac\rho{\phi(q)} + \frac12(\phi(q)-1) \frac\rho{\phi(q)} < \rho,
\end{align*}
which shows that $X_W^R \in \cB_\rho\big(\bt-u(1,\dots,1) \big)\subset \cC_\rho(\bt)$, implying equation~\eqref{PApositive2}.

We have reduced the proposition to showing that each factor on the right-hand side of equation~\eqref{PApositive2} is positive.
Doing so for the first set of factors is easy: for each robust character $\chi\mod q$, thanks to equation~\eqref{inequality2}, we may apply Lemma~\ref{robot variable lemma} with $z=v_\chi-iw_\chi$, and with $\Gamma = \Gamma^S(\chi) \cap (0,T]$ and $\lambda_\gamma = 2/\sqrt{1/4+\gamma^2}$, to immediately conclude that $P \big(|  U_\chi^S - (v_\chi+iw_\chi) | < \varepsilon \big)$ is positive for any $\varepsilon>0$.
On the other hand, the positivity of each factor in the second product on the right-hand side of equation~\eqref{PApositive2} is an easy application of Chebyshev's inequality~\cite[page~47]{Sh} to the complex-valued random variables $V_\chi$: for each robust character $\chi\mod q$, the variance $\sigma^2(V_\chi^S)$ satisfies
\[
\sigma^2(V_\chi^S) = 4 \sum_{\substack{\gamma\in\Gamma^S(\chi) \\ \gamma > T}} \frac{\sigma^2(Z_\gamma)}{{\frac14+\gamma^2}} = 4 \sum_{\substack{\gamma\in\Gamma^S(\chi) \\ \gamma > T}} \frac1{{\frac14+\gamma^2}}
\]
(since the $Z_\gamma$ are mutually independent), and thus
\begin{equation}  \label{apply Chebyshev}
   P \bigg(  |V_\chi^S| \ge \frac{\rho}{4 \sqrt{r} \phi(q)} \bigg)
    \le \bigg(\frac{16r\phi(q)^2}{\rho^2} \bigg) \sigma^2 (V_\chi^S) = 
    \bigg(\frac{16r\phi(q)^2}{\rho^2} \bigg) \cdot 4
    \sum_{\substack{\gamma\in\Gamma^S(\chi) \\ \gamma>T}} \frac1{\frac14+\gamma^2} < 1,
\end{equation}
where the last inequality follows from equation~\eqref{inequality1}.
In other words, $P\big(  |V_\chi^S| <  {\rho}/{4 \sqrt{r} \phi(q)} \big)$ is indeed positive.
\end{proof}

\begin{cor}  \label{omnipresent}
For every integer $q\ge3$ and every real number $V >0$, there exists a constant $W(q,V)$ with the following property: Choose $W \ge W(\cQ,V)$, and suppose that the set of vectors
\[
\big\{ (1,\dots,1) \big\} \cup \big\{ \bx_\chi \colon \chi\mod q \text{ is $W$-robust} \big\} \cup \big\{ \by_\chi \colon \chi\mod q \text{ is $W$-robust} \big\}
\]
spans $\R^r$, where~$\bx_\chi$ and~$\by_\chi$ are as defined in Definition~\ref{ETx def}.
Then for every $\bt \in \R^{r}$ such that $\|\bt\| \le V$, we have
$\mu^R_W(\cC_\rho(\bt)) > 0$, where $\mu^R_W$ is the measure defined in Proposition~\ref{muconvolutionrobust}.
\end{cor}

\begin{proof}
Define $W(q,V) = \max_{\cQ} W(\cQ,V)$ in the notation of Proposition~\ref{omnipresent temp try}, where the maximum is taken over all nonempty subsets $\cQ$ of nonprincipal characters\mod q. Then apply Proposition~\ref{omnipresent temp try} with $\cQ = \{\chi\mod q\colon \chi$ is $W(q,V)$-robust$\}$ and any $\|\bt\| \le V$ (which is valid by the assumption of this proposition).
\end{proof}

\begin{prop}  \label{omnipresent temp try, strong}
For every nonempty set of nonprincipal characters $\cQ\mod q$ and every real number $V>0$, there exists a constant $W(\cQ,V)$ with the following property:
Suppose that $\bt \in \R^{r}$ satisfies $\|\bt\| \le V$ and is in the span of the set of vectors
\begin{equation*}
\big\{ \bx_\chi \colon \chi\in\cQ \big\} \cup \big\{ \by_\chi \colon \chi\in\cQ \big\}.
\end{equation*}
If $W \ge W(\cQ,V)$ is a real number such that every $\chi\in\cQ$ is $W$-robust, then $\mu^R_W(\cB_\rho(\bt)) > 0$ for every $\rho >0$, where $\mu^R_W$ is the measure defined in Proposition~\ref{muconvolutionrobust}.
\end{prop}

\begin{proof}
By the spanning hypothesis, there exist real numbers $\{ v_\chi \colon \chi\in\cQ \}$ and $\{ w_\chi \colon \chi\in\cQ \}$ such that
\begin{equation*}
\bt = \sum_{ \chi\in\cQ} ( 2v_\chi \bx_\chi + 2w_\chi \by_\chi ),
\end{equation*}
which is exactly the same as equation~\eqref{spanning hypothesis} with $u=0$. The entire proof of Proposition~\ref{omnipresent temp try} goes through unchanged, resulting in the lower bound~\eqref{PApositive2} for $P\big( X_W^R \in \cB_\rho\big(\bt-u(1,\dots,1) \big) \big) = P( X_W^R \in \cB_\rho(\bt))$ which leads immediately the desired conclusion.
\end{proof}

\begin{cor}  \label{omnipresent, strong}
For every integer $q\ge3$ and every real number $V >0$, there exists a constant $W(q,V)$ with the following property:
Choose $W \ge W(\cQ,V)$, and suppose that $\bt \in \R^{r}$ satisfies $\|\bt\| \le V$ and is in the span of the set of vectors
\[
\big\{ \bx_\chi \colon \chi\mod q \text{ is $W$-robust} \big\} \cup \big\{ \by_\chi \colon \chi\mod q \text{ is $W$-robust} \big\}.
\]
Then $\mu^R_W(\cB_\rho(\bt)) > 0$.
\end{cor}

\begin{proof}
Define $W(q,V) = \max_{\cQ} W(\cQ,V)$ in the notation of Proposition~\ref{omnipresent temp try, strong}, where the maximum is taken over all nonempty subsets $\cQ$ of nonprincipal characters\mod q. Then apply Proposition~\ref{omnipresent temp try, strong} with $\cQ = \{\chi\mod q\colon \chi$ is $W(q,V)$-robust$\}$ and any $\|\bt\| \le V$ (which is valid by the assumption of this proposition).
\end{proof}

We are now prepared to establish the second part of our main theorem. For every permutation $\sigma$ of $\{1,2,\dots,r\}$, define
\begin{equation} \label{S sigma def}
\cS_\sigma = \big\{ (x_1,\dots,x_r) \in \R^r \colon x_{\sigma(1)} < x_{\sigma(2)} < \cdots < x_{\sigma(r)} \big\}.
\end{equation}

\begin{proof}[Proof of Theorem~\ref{inclusive theorem}(b)]
Let $\sigma$ be any permutation of $\{1,\dots,r\}$.
Choose $0<\rho<1$ and define $\bt = \bt(\sigma) = {-}(\sigma^{-1}1, \dots, \sigma^{-1}r)$; we note for later use that $\|\bt\| \le r^2 \le \phi(q)^2$. We begin by showing that $\mu\big( \cC_\rho(\bt) \big) > 0$.

Let $W>0$ be a parameter to be chosen later.
By Proposition~\ref{existence of muN prop} there exist probability measures $\mu^R_W$ and $\mu^N_W$ such that
$\mu= \mu^R_W* \mu^N_W$.  
By Proposition~\ref{has mass somewhere small}, there exist $\beta(q)>0$ and $\bn\in\R^r$ such that $\|\bn\| \le \beta(q)$ and
\begin{equation}
  \label{muNBpos}
\mu^N_W(\cB_{\rho/2}(\bn)) > 0.
\end{equation}
Set $V = \beta(q) + \phi(q)^2$ and $\by=\bt-\bn$, and note that $\|\by\| \le \|\bt\|+\|\bn\| \le V$.

If $\bu\in \cC_{\rho/2}(\by)$ and $\bv\in \cB_{\rho/2}(\bn)$, then $\bu+\bv \in \cC_\rho(\bt)$ by the triangle inequality.
In other words, $\cC_{\rho/2}(\by) + \cB_{\rho/2}(\bn) \subset \cC_\rho(\bt)$ as sets, and thus from equation~\eqref{convolution lower bound} we obtain
\begin{equation}
  \label{muClowerbound}
\mu(\cC_\rho(\bt)) \ge \mu^R_W(\cC_{\rho/2}(\by)) \mu^N_W(\cB_{\rho/2}(\bn)).
\end{equation}

We are assuming that
\[
\big\{ (1,\dots,1) \big\} \cup \big\{ \bx_\chi \colon \chi\mod q \text{ is $W$-robust} \big\} \cup \big\{ \by_\chi \colon \chi\mod q \text{ is $W$-robust} \big\}
\]
spans $\R^r$.
If we also assume that $W\ge W(q,V)$, then since $\|\by\| \le V$, by Corollary~\ref{omnipresent} we thus have $\mu^R_W(\cC_{\rho/2}(\by)) > 0$. Since also $\mu^N_W(\cB_{\rho/2}(\bn)) > 0$ by equation~\eqref{muNBpos}, it follows that $\mu(\cC_\rho(\bt)) >0$ as desired.

If we further assume that $W \ge (2\phi(q)+1)/\min \Gamma(q)$, we deduce that each $W$-robust character is $(2r+1)$-sturdy. In particular, the set of vectors 
\begin{multline*}
\big\{ (1,\dots,1) \big\} \cup \big\{ \big( \Re\chi(a_1), \dots, \Re\chi(a_r) \big) \colon \chi\mod q \text{ is $(2r+1)$-sturdy} \big\} \\
\cup \big\{ \big( \Im\chi(a_1), \dots, \Im\chi(a_r) \big) \colon \chi\mod q \text{ is $(2r+1)$-sturdy} \big\}
\end{multline*}
spans the vector space~$\R^r$; therefore by the proof of Theorem~\ref{inclusive theorem}(a), the logarithmic density of the set of~$x$ satisfying the inequalities~\eqref{chain of inequalities} exists and equals $\mu(\cS_\sigma)$. On the other hand, by our choice of $\bt$ and $\rho$, we have $\cC_\rho(\bt) \subset \cS_\sigma$, and therefore $\mu(\cC_\rho(\bt)) >0$ implies that $\mu\big( \cS_\sigma \big) > 0$.

In summary, if we set $W(q) = \max\{ (2\phi(q)+1)/\min \Gamma(q), W(q,V) \}$, then under the assumptions of the theorem, the logarithmic density of the set of~$x$ satisfying the inequalities~\eqref{chain of inequalities} exists and is positive, as desired.
\end{proof}

In Theorem~\ref{inclusive theorem}(b), we assumed that the set of vectors
\begin{multline*}
\big\{ (1,\dots,1) \big\} \cup \big\{ \big( \Re\chi(a_1), \dots, \Re\chi(a_r) \big) \colon \chi\mod q \text{ is $W$-robust} \big\} \\
\cup \big\{ \big( \Im\chi(a_1), \dots, \Im\chi(a_r) \big) \colon \chi\mod q \text{ is $W$-robust} \big\}
\end{multline*}
spans the vector space~$\R^r$. On the other hand, our proof requires only a seemingly weaker statement, namely that the subspace of $\R^r$ spanned by
\begin{multline*}
\big\{ \big( \Re\chi(a_1), \dots, \Re\chi(a_r) \big) \colon \chi\mod q \text{ is $W$-robust} \big\} \\
\cup \big\{ \big( \Im\chi(a_1), \dots, \Im\chi(a_r) \big) \colon \chi\mod q \text{ is $W$-robust} \big\}
\end{multline*}
intersects every wedge $\cS_\sigma$, defined in equation~\eqref{S sigma def},
as $\sigma$ runs over all permutations of the set $\{1,2,\dots,r\}$. However, it follows from the next proposition that these two statements are actually equivalent.

\begin{prop}  \label{need all these directions prop}
Let $r\ge2$, and let $\cV$ be a subspace of $\R^r$ not containing $(1,\dots,1)$ that intersects every wedge $\cS_\sigma$. Then $\dim \cV = r-1$. In particular, $\cV\cup(1,\dots,1)$ spans $\R^r$.
\end{prop}

\begin{proof}
We prove the following equivalent statement: if $\cV$ is a subspace of $\R^r$ not containing $(1,\dots,1)$ and $\dim \cV \le r-2$, then $\cV$ does not intersect some wedge $\cS_\sigma$. The orthogonal complement (under the standard inner product) of the subspace generated by $\cV$ and $(1,\dots,1)$ has dimension $r-(\dim \cV+1)\ge1$; therefore we may choose a nonzero vector $\by=(y_1,\dots,y_r) \in \big(\cV\cup(1,\dots,1)\big)^\perp$, so that in particular, $y_1+\cdots+y_r=0$. By permuting the coordinates of $\R^r$, we may assume that $y_1\le y_2\le \cdots \le y_r$; note that the assumption that $\by$ is nonzero implies that $y_1<0$ and $y_r>0$. We claim that $\cV$ does not intersect $\cS = \big\{ (x_1,\dots,x_r) \in \R^r \colon x_1 < x_2 < \cdots < x_r \big\}$.

Suppose for the sake of contradiction that there was some $(z_1,\dots,z_r) \in \cS \cap \cV$. Then $z_1<z_2<\cdots<z_r$, and $y_1z_1 + y_2z_2 + \cdots + y_rz_r=0$ since $(z_1,\dots,z_r) \in \cV$ and $(y_1,\dots,y_r) \in \cV^\perp$. Suppose that $k$ and $\ell$ are chosen so that $y_k$ is the largest negative value, and $y_\ell$ the smallest positive value, among $\{y_1,\dots,y_r\}$; in other words, $y_k < 0 = y_{k+1} = \cdots = y_{\ell-1} < y_\ell$, where it is possible that $\ell=k+1$. Then, thanks to the known ordering $z_1<z_2<\cdots<z_r$, we have
\begin{align*}
0 = y_1z_1 + y_2z_2 + \cdots + y_rz_r &\ge (y_1+\cdots+y_k)z_k + (y_\ell+\cdots+y_r)z_\ell \\
&= (z_\ell-z_k)(y_\ell+\cdots+y_r) > 0,
\end{align*}
where the middle equality follows from $y_1+\cdots+y_r=0$. This contradiction establishes the proposition.
\end{proof}

We proceed to establish the last part of our main theorem, using a method very similar to that used in the proof of part~(b). Recall the vector-valued normalized error term~$E(x)$ is defined in equation~\eqref{Ey}.

\begin{proof}[Proof of Theorem~\ref{inclusive theorem}(c)]
Choose $\rho>0$ and $\bx \in \R^r$; if $r=\phi(q)$, then 
we assume that $\bx$ is an element of the hyperplane $x_{1} + \cdots + x_{\phi(q)} =0$. 
We need to show that $\mu\big(B_{\rho}(\bx)\big)$ exists and is positive; we begin by establishing the existence.
Let $W>0$ be a parameter to be chosen later.
If we assume that $W\ge (2\phi(q)+1)/\min \Gamma(q)$, we deduce that each $W$-robust character is $(2r+1)$-sturdy. In particular, the set of vectors 
\begin{multline*}
  \big\{ \big( \Re\chi(a_1), \dots, \Re\chi(a_r) \big) \colon \chi\mod q \text{ is $(2r+1)$-sturdy} \big\} \\
\cup \big\{ \big( \Im\chi(a_1), \dots, \Im\chi(a_r) \big) \colon \chi\mod q \text{ is $(2r+1)$-sturdy} \big\}
\end{multline*}
spans~$\R^r$ if $r < \phi(q)$, or spans $\big\{x_1 + \cdots + x_{\phi(q)}=0 \big\}$ if $r=\phi(q)$.
We can then show that the logarithmic density $\delta\big( \{ t\ge1\colon E(y) \in \cB_\rho(\bx) \} \big)$ exists and equals $\mu(\cB_{\rho}(\bx))$ using a slight variant of the proof of Theorem~\ref{inclusive theorem}(a) that invokes Lemma~\ref{conv vanish sphere}
instead of Lemma~\ref{conv vanish subsp}.

We now establish that  $\mu(B_{\rho}(\bx)) >0$.
As before, by Proposition~\ref{existence of muN prop} there exist probability measures $\mu^R_W$ and $\mu^N_W$ such that $\mu= \mu^R_W* \mu^N_W$.
By Lemma~~\ref{has mass somewhere small}, we may select 
$\bn \in \R^r$ with $\|\bn\| \le \beta(q)$, which also is an element of the subspace 
$x_1+\cdots +x_{\phi(q)}=0$ in the case $r=\phi(q)$, such that
\begin{equation}
   \mu^N_W(\cB_{\rho/2}(\bn)) > 0.
\end{equation} 
Let $\by = \bx-\bn$, and set $V = \|\bx\| + \beta(q)$, so that $\|\by\| \le V$.  
If we further assume that $W \ge W(q,V)$, where $W(q,V)$ is given by 
Corollary~\ref{omnipresent, strong}, then it follows from that corollary that $\mu^R_W(\cB_{\rho/2}(\by)) >0$. 
Therefore by equation \eqref{convolution lower bound} we deduce that  $\mu(\cB_{\rho}(\bx) ) > 0$ since $  \cB_{\rho/2}(\by) +\cB_{\rho/2}(\bn) \subset \cB_{\rho}(\bx)$.

In summary, if we set $W(q) = \max\{ (2\phi(q)+1)/\min \Gamma(q), W(q,V) \}$, then under the assumptions of the theorem, the logarithmic density $\delta\big( \{ t\ge1\colon E(y) \in \cB_\rho(\bx) \} \big)$ exists and is positive, as desired.
\end{proof}

Finally, we provide the analogous proofs for Theorem~\ref{pili theorem}(b)--(c).

\begin{proof}[Proof of Theorem~\ref{pili theorem}(b)]
Let $\mu_{\pi}= \mu^R_{\pi}* \mu^N_{\pi}$ be the probability measures defined in the proof of Theorem~\ref{pili theorem}(a); we begin by showing that $\mu_\pi\big( (0,\infty) \big)>0$.
By a variant of Proposition~\ref{has mass somewhere small}, there exists $n\in\R$ such that
\begin{equation} \label{W0}
\mu^N_{\pi}\big( (n-1,n+1) \big) > 0.
\end{equation}
Since $u+v>0$ whenever $u>1-n$ and $v>n-1$, we see from equation~\eqref{convolution lower bound} that
\begin{equation}
  \label{muClowerbound pi}
\mu_\pi\big( (0,\infty) \big) \ge \mu_\pi^R\big( (1-n,\infty) \big) \mu_\pi^N\big( (n-1,n+1) \big) \gg \mu_\pi^R\big( (1-n,\infty) \big)
\end{equation}
by equation~\eqref{W0}.
Set 
\begin{equation} \label{W1}
W(1) = \max\{8,|2-n|\} + \frac{1}{4} \sum_{ \gamma\in\Gamma^S(\chi_0) } \frac{1}{\gamma^3}
\end{equation}
so that we are assuming that $\sum_{\gamma\in\Gamma^S(\chi_0)} \frac1\gamma > W(1)$, and therefore that
\[
\sum_{\substack{\gamma\in\Gamma^S(\chi_0) \\ \gamma\le T}} \frac1\gamma > \frac{W(1)}2
\]
when~$T$ is sufficiently large.
As in the proof of Proposition~\ref{omnipresent temp try},
\begin{equation} \label{zs1}
\sum_{\substack{\gamma\in\Gamma^S(\chi) \\ \gamma\le T}} \frac2{\sqrt{\frac14+\gamma^2}} \ge \sum_{\substack{\gamma\in\Gamma^S(\chi) \\ \gamma\le T}} \frac2\gamma - \frac14 \sum_{\substack{\gamma\in\Gamma^S(\chi) \\ \gamma\le T}} \frac{1}{\gamma^3} > W(1) - \frac14 \sum_{\gamma\in\Gamma^S(\chi)} \frac{1}{\gamma^3} = \max\{8,|2-n|\}.
\end{equation}
Furthermore, when~$T$ is sufficiently large, we also have
\begin{equation} \label{zs2}
  \sum_{\substack{\gamma\in\Gamma^S(\chi_0) \\ \gamma>T}} \frac1{\frac14+\gamma^2} < \frac1{16}.
\end{equation}
Using such a value of~$T$, we write the random variable~$X_{\pi}^R$ from equation~\eqref{Xpi definition} as 
$X_{\pi}^R = U_{\pi}^R + V_{\pi}^R$, where
\begin{equation}
 \label{UpiRVpiR}
  U_{\pi}^R = 2 \Re  \sum_{\substack{\gamma \in \Gamma^S(\chi_0) \\ \gamma\le T}} \frac{Z_\gamma}{\sqrt{\frac14+\gamma^2}}
  \quad\text{and}\quad
  V_{\pi}^R = 2 \Re  \sum_{\substack{\gamma \in \Gamma^S(\chi_0) \\ \gamma > T}} \frac{Z_\gamma}{\sqrt{\frac14+\gamma^2}}.
\end{equation}
By the inequality~\eqref{zs1}, Lemma~\ref{robot variable lemma} implies that the event $|U_{\pi}^R+n-2| < \frac12$ occurs with positive probability. Furthermore, by the inequality~\eqref{zs2}, an application of Chebyshev's inequality analogous to equation~\eqref{apply Chebyshev} shows that there is a positive probability that $|V_{\pi}^R| < \frac12$. Consequently, by the independence of the random variables 
$U_{\pi}^R$ and $V_{\pi}^R$, 
 there is a positive probability that $|X_{\pi}^R +n-2| < \frac12+\frac12=1$, which shows that $\mu_{\pi}^R\big( (1-n,\infty) \big) \ge \mu_{\pi}^R\big( (1-n,3-n) \big) > 0$. Given the lower bound~\eqref{muClowerbound pi}, we conclude that $\mu_{\pi}\big( (0,\infty) \big) >0$ as desired.

Since  $\sum_{\gamma\in\Gamma^S(\chi_0)} \frac1\gamma \ge 8 
$, and the least element of $\Gamma^S(\chi_0)$ is $14.134\dots$, it follows that 
there exist more than three self-sufficient ordinates. Therefore, following the argument 
in Theorem~\ref{pili theorem}(a), we obtain the existence of the limit defining the density
\[
  \delta(\{ x \ge 2 \colon \pi(x) > \li(x) \} =  \mu_{\pi}\big( (0,\infty) \big) > 0.
\]

An almost identical argument shows that $\mu_\pi\big( (-\infty,0) \big)>0$, and therefore the race between $\pi(x)$ and $\li(x)$ is inclusive.
\end{proof}

\begin{proof}[Proof of Theorem~\ref{pili theorem}(c)]
We amalgamate the methods used in the proofs of Theorem~\ref{inclusive theorem}(c) and Theorem~\ref{pili theorem}(b).
Let $y \in \R$ and $\rho >0$, so that we need to show that $\mu_{\pi}\big( (y-\rho,y+\rho) \big)>0$; the existence of this interval measure is justified exactly as at the end of the proof of Theorem~\ref{pili theorem}(c).
As in the previous proof, $\mu_{\pi}= \mu_{\pi}^R*\mu_{\pi}^N$ and there exists $n \in\R$ such that $\mu^N_{\pi}\big( (n-\frac\rho2,n+\frac\rho2) \big) > 0$.
Since $\sum_{\gamma\in\Gamma^S(\chi_0)} \frac1\gamma$
diverges it follows that there exists $T$ such that 
\begin{equation}
   \label{zs1b}
     \sum_{\substack{\gamma\in\Gamma^S(\chi_0) \\ \gamma\le T}} \frac2{\sqrt{\frac14+\gamma^2}} >  \max\{8, |y| + |n|\}
  \text{ and }
  \sum_{\substack{\gamma\in\Gamma^S(\chi_0) \\ \gamma>T}} \frac1{\frac14+\gamma^2} < \frac{\rho^2}{16}.
\end{equation}
As before, $\mu_\pi^R$ is the probability measure associated to the random variable $X_\pi^{R} = U_{\pi}^R + V_{\pi}^R$ as in equation~\eqref{UpiRVpiR}.  Since $|y-n| \le |y| + |n|$, the first inequality in equation~\eqref{zs1b} and Lemma~\ref{robot variable lemma} imply that there is a positive probability that $|U_{\pi}^R - (y-n)| < \frac{\rho}{4}$, while the second inequality in~\eqref{zs1b} and Chebyshev's inequality imply that there is a positive probability that $|V_{\pi}^R| < \frac{\rho}4$.  Consequently, by the independence of the random variables 
$U_{\pi}^R$ and $V_{\pi}^R$, 
the event $|X_{\pi}^R - (y-n)| < \frac{\rho}4+\frac{\rho}4=\frac{\rho}{2}$ has positive probability and thus $
\mu^R_{\pi}\big( (y-n-\frac\rho2,y-n+\frac\rho2) \big) > 0$.
It then follows from equation~\eqref{convolution lower bound} that 
\[
\mu_\pi\big( (y-\rho,y+\rho) \big) \ge \mu^R_{\pi}\big( (y-n-\tfrac\rho2,y-n+\tfrac\rho2) \big) \mu^N_{\pi}\big( (n-\tfrac\rho2,n+\tfrac\rho2) \big) >0,
\]
completing the proof. 
\end{proof}

The methods in our article extend to a wide class of prime number races. 
The main feature we require is an ``explicit formula" relating the functions under 
consideration to the zeros of some $L$-functions (not necessarily degree-$1$ $L$-functions, in general). 
For instance, we can establish the following result regarding the race between $\pi(x;q,a)$ and $\li(x)/\phi(q)$ (for which the principal character~$\chi_0$ sheds its irrelevance)
and the race between $\pi(x;q,a)$ and $\pi(x)/\phi(q)$:

\begin{theorem}
\label{pili theorem for APs}
Assume GRH. Let $a$ be a reduced residue modulo~$q$.
\begin{enumerate}
\item If $\Gamma^S(q) \cup \Gamma^S(\chi_0)$ has at least three elements, then the race between $\pi(x;q,a)$ and $\li(x)/\phi(q)$ is weakly inclusive.
\item There exists a constant $W(q)$ such that if $\sum_{\gamma\in\Gamma^S(q) \cup \Gamma^S(\chi_0)} \frac1\gamma > W(q)$, then the race between $\pi(x;q,a)$ and $\li(x)/\phi(q)$ is inclusive.
\item If the sum $\sum_{\gamma\in\Gamma^S(q) \cup \Gamma^S(\chi_0)} \frac1\gamma$ diverges, then the race between $\pi(x;q,a)$ and $\li(x)/\phi(q)$ is strongly inclusive.
\item If $\Gamma^S(q)$ has at least three elements, then the race between $\pi(x;q,a)$ and $\pi(x)/\phi(q)$ is weakly inclusive.
\item There exists a constant $W(q)$ such that if $\sum_{\gamma\in\Gamma^S(q)} \frac1\gamma > W(q)$, then the race between $\pi(x;q,a)$ and $\pi(x)/\phi(q)$ is inclusive.
\item If the sum $\sum_{\gamma\in\Gamma^S(q)} \frac1\gamma$ diverges, then the race between $\pi(x;q,a)$ and $\pi(x)/\phi(q)$ is strongly inclusive.
\end{enumerate}
\end{theorem}
Since the proof of this theorem is very similar to the proofs of Theorem~\ref{pili theorem},
we will not include the proof.  We just note that 
parts (a)  and (d) follow from the arguments of Section~\ref{log densities section}
and the other parts follow from the arguments of  Section~\ref{cylinder section}, using the explicit
formulae
\begin{align*}
  \frac{\log x}{\sqrt{x}} \bigg( \pi(x;q,a) - \frac{\li(x)}{\phi(q)} \bigg) &= 
  -\frac{c(q,a)}{\phi(q)} -\frac{1}{\phi(q)}
    \sum_{\chi\mod q} 
     \overline{\chi}(a) \sum_{\substack{\gamma \in \mathbb{R} \\ L(\frac{1}{2}+i \gamma,\chi)=0} }
    \frac{x^{i \gamma}}{\frac{1}{2}+i \gamma} + O \bigg( \frac{1}{\log x} \bigg),
\\
    \frac{\log x}{\sqrt{x}} \bigg( \pi(x;q,a) - \frac{\pi(x)}{\phi(q)}  \bigg) &= -\frac{c(q,a)}{\phi(q)} -\frac{1}{\phi(q)}
   \sum_{\substack{\chi\mod q \\ \chi \ne \chi_0}} 
    \overline{\chi}(a) \sum_{\substack{\gamma \in \mathbb{R} \\ L(\frac{1}{2}+i \gamma,\chi)=0} }
    \frac{x^{i \gamma}}{\frac{1}{2}+i \gamma} + O \bigg( \frac{1}{\log x} \bigg).
\end{align*}

\setcounter{section}{0}
\renewcommand{\thesection}{\Alph{section}}

\section{Appendix: Some probability}  \label{probability appendix}

In this section we record some facts from probability that are required in this article. 

\begin{definition}  \label{pushforward def}
Let $\Omega_1$ and $\Omega_2$ be probability spaces, and let $\Psi\colon \Omega_1\to\Omega_2$ be a measurable function. For any probability measure $\mu$ on $\Omega_1$, there exists a probability measure $\Psi_*\mu$ on $\Omega_2$, called the {\em pushforward of $\mu$ under $\Psi$}, defined by
\[
\Psi_*\mu(\cB) = \mu(\Psi^{-1}(\cB))
\]
for all measurable subsets $\cB$ of $\Omega_2$,
and with the property that
\begin{equation}  \label{pushforward change of variables}
\int_{\Omega_2} g(y) \,d\Psi_*\mu(y) = \int_{\Omega_1} g(\Psi(x)) \,d\mu(x)
\end{equation}
for any measurable function $g(y)$ defined on $\Omega_2$ (see \cite[Section 6, Theorem~7]{Sh} or \cite[Theorem~16.13]{Bi}).
\end{definition}

\noindent In this article, we deal with probability measures that are convolutions.

\begin{lemma} \label{weakconv}
Let $(\mu_n)$ and $(\nu_n)$ be sequences of probability measures on $\R^r$ which converge weakly to probability measures $\mu$ and $\nu$ on $\R^r$, respectively. 
Then $(\mu_n*\nu_n)$ converges weakly to $\mu*\nu$. 
\end{lemma}

\begin{proof}[Proof]
Since $\mu_n \to \mu$ weakly and $\nu_n \to \nu$ weakly, it follows from 
Levy's theorem \cite[p. 383]{Bi}, 
 a criterion for weak convergence in terms of characteristic functions, that
$\widehat\mu_n(\bt) \to   \widehat\mu(\bt) \text{ and }   \widehat\nu_n(\bt) \to   \widehat\nu(\bt)$ for all $\bt \in \R^r$.
By the definition of convolution it follows that 
$\widehat{\mu_n*\nu_n}(\bt) = \widehat\mu_n(t) \widehat\nu_n(\bt)$ 
and thus $\widehat{\mu_n*\nu_n}(\bt)  \to  \widehat\mu(\bt) \widehat\nu(\bt) = \widehat{\mu*\nu}(\bt)$ as $n \to \infty$, for all $\bt \in \R^r$.  
Hence, again by Levy's theorem, we see that $\mu_n*\nu_n \to \mu*\nu$ weakly. 
\end{proof} 

Another way to establish weak convergence is via tightness.
\begin{definition}  \label{tightdefn}
A sequence $(\mu_n)$ of probability measures on $\R^r$ 
is {\it tight} if, for every $\varepsilon >0$, there is a compact set $\cK \subset \R^r$ such that the measure of its complement is uniformly small, that is, if $\mu_n(\cK^c) \le \varepsilon$ for all $n\ge1$. 
\end{definition}

The following result explains the significance of a tight sequence of probability measures. For a proof, see~\cite[Theorem~11.3]{Br} or~\cite[Theorem~1, pages~318--320]{Sh}.
 
\begin{theorem}  \label{tightsubsequence}
Let $(\mu_n)$ be a tight sequence of probability measures on $\R^r$.   Then there exists 
a probability measure $\mu$ on $\R^r$ and a subsequence $(\mu_{n_k})$
such that $\mu_{n_k}$ converges weakly to~$\mu$. 
\end{theorem}

We now present a criterion for showing that a sequence of probability measures is tight:

\begin{theorem} \label{tightnesscriterion}
Let $(\mu_n)$ be a sequence of probability measures on $\R^r$,
and let $(\widehat\mu_n)$ denote the corresponding characteristic functions.
Assume that there exists a neighbourhood $\cN$ of the origin such that
\begin{itemize}
 \item[(i)]  $h(\bt) := \lim_{n \to \infty} \widehat\mu_n(\bt)$ exists for every $\bt \in \cN$;
 \item[(ii)]  $h(\bt)$ is continuous at the origin. 
\end{itemize}
Then $(\mu_n)$ is tight. 
\end{theorem}

\noindent
This criterion is a minor variant of the result~\cite[Theorem~11.6, p. 236]{Br}, except that our condition~(i) assumes only that $\lim_{n \to \infty} \widehat\mu_n(\bt)$ exists in a neighbourhood of the origin rather than on all of~$\R^r$. The proof under this weaker hypothesis is essentially the same; we remark that one must repair a minor error in the proof in~\cite{Br}, by interpreting the complement $\cK(u)^c$ correctly as $\{x\in\R^r\colon \max_{1\le j\le r} |x_j| > u^{-1}\}$ rather than as $\{x\in\R^r\colon \min_{1\le j\le r} |x_j| > u^{-1}\}$ (indeed, the repaired proof demonstrates the philosophy that tightness of a sequence of probability measures on $\R^r$ reduces to tightness in each coordinate separately).

To finish this section, we give a criterion for when a probability measure is absolutely continuous with respect to Lesbesgue
measure.  

\begin{lemma}  \label{density}
Let $\mu$ be a probability measure on $\R^n$ with characteristic function $\widehat\mu$.   \\
\begin{enumerate}
\item Let $(\ba,\bb]$ be an interval of continuity of $\mu$.   Then 
\begin{equation}
 \label{muab}
  \mu((\ba,\bb]) = \lim_{c \to \infty} \frac{1}{(2 \pi)^n} \int_{-c}^{c}
  \prod_{j=1}^{n} \frac{e^{-ita_j}-e^{-itb_j} }{it_j}
  \widehat\mu(t_1,\ldots, t_n) \,dt_1 \cdots \,dt_n. 
\end{equation}
\item If $\int_{\R^n} |\widehat\mu(\bt)| \,d\bt < \infty$,
then $\mu$ possesses a Lebesgue-integrable density $g$, so that
\[
  \mu(\cB) = \int_{\cB} g(\bx) \,d\bx
\]
for any Borel subset~$\cB$ of~$\R^n$, where~$d\bx$ is Lebesgue measure on~$\R^n$. In particular,~$\mu$ is absolutely continuous with respect to Lebesgue measure on~$\R^n$. 
\end{enumerate}
\end{lemma}

A proof of the $n=1$ case of part~(a) can be found in Shiryaev~\cite[Theorem~3(a), page~283]{Sh}; the general case is relegated to an exercise~\cite[page~297]{Sh}, although the negative signs in the exponents were unintentionally omitted in its statement. 
This statement can also be found in~\cite[equation~(29.3)]{Bi} and~\cite[equation 10.6.2]{Cr}. The $n=1$ case of part~(b)
can be found in Lukacs~\cite[Theorem~3.2.2]{Lu}, and the general case can be proven similarly.

\section{Appendix: The Kronecker--Weyl equidistribution theorem}  \label{KW appendix}

The Kronecker--Weyl equidistribution theorem (see \cite[pages~12--13]{Hlawka}) is a classical theorem in Diophantine approximation.  It asserts that if $k$ real numbers  $\{ \xi_1, \ldots, \xi_k \}$ are linearly independent
over $\Q$, then the one-parameter 
subgroup $\{ t(\xi_1, \ldots, \xi_k ) \colon t \in \R \}$ is equidistributed in $\T^k$.
The theorem can be generalized to the case that $\{ \xi_1, \ldots, \xi_k \}$ is linearly dependent over~$\Q$. 
In this case, it turns out that the one-parameter 
subgroup $\{ t(\xi_1, \ldots, \xi_k ) \colon t \in \R \}$ becomes equidistributed in a subtorus of~$\T^k$.
Note that this generalization, a proof of which can be found in~\cite[Theorem~4.2]{D}, is a special case of Ratner's famous equidistribution theorem~\cite[Theorem~(1.3.4)]{M1}.

\begin{lemma} \label{kw}
Let $\xi_1, \ldots, \xi_k$ be real numbers.
There exists a subtorus $\cA$ of $\T^k$ such that for all continuous functions~$f$ defined on~$\T^k$,
\[
   \lim_{y \to \infty} \frac1y \int_0^y f( t \xi_1, \ldots, t\xi_k) \,dt = \int_{\cA} f(\ba) \,d\ba,
\]
where $d\ba$ is Haar measure on $\cA$.  
\end{lemma}

It can be seen from the proof that the ray $\{(t\xi_1,\dots,t\xi_k)\colon t\in\R\}$ is contained in~$\cA$; taking~$f$ to be supported on a small ball around any point of $\cA$, we see that there are points in this ray arbitrarily close to any point of~$\cA$. Therefore the subtorus $\cA$ is the closure of the ray $\{(t\xi_1,\dots,t\xi_k)\colon t\in\R\}$ inside $\T^k$.

We now record a simple corollary to Lemma~\ref{kw}.

\begin{cor} \label{kwcorollary}
Let $k$ and $r$ be positive integers, let $\xi_1, \ldots, \xi_k$ be real numbers, and let $\bz_1,\dots,\bz_k\in\C^r$. 
Let $\Psi\colon \T^k \to \R^r$ be defined by 
\[ 
  \Psi(\zeta_1, \ldots, \zeta_k)  = 2 \Re \bigg(  \sum_{j=1}^{k} \bz_j e^{2 \pi i \zeta_j} \bigg).
\]
Then the function $\eta(t)=  \Psi(t \xi_1, \ldots, t\xi_k)$
possesses a limiting distribution. More precisely:
\begin{enumerate}
\item there exists a subtorus $\cA$ of $\T^k$ such that 
\[
   \lim_{y \to \infty} \frac1y \int_0^y f ( \eta(t) ) \,dt = \int_{\cA} (f \circ \Psi)(\ba) \,d\ba
\]
for all bounded continuous functions $f \colon \R^r \to \R$, where $d\ba$ is Haar measure on $\cA$;
\item there exists a probability measure $\nu$ on $\R^r$ such that 
\[
   \lim_{y \to \infty} \frac1y \int_0^y f ( \eta(t) ) \,dt = \int_{\R^r} f(\bx) \,d \nu(\bx)
\]
for all bounded continuous functions $f \colon \R^r \to \R$.
\end{enumerate}
\end{cor}

\begin{proof}
Let $f \colon \R^r \to \R$ be a bounded continuous function.
Applying Lemma~\ref{kw} with $f \circ \Psi$ in place of~$f$,
\begin{equation}
\begin{split}
  \label{ld1}
   \lim_{y \to \infty} \frac1y \int_0^y f ( \eta(t) ) \,dt = \lim_{y \to \infty} \frac1y \int_0^y f(\Psi(t\xi_1, \ldots, t\xi_k)) \,dt & = \int_{A} (f \circ \Psi) (\ba) \,d\ba
\end{split}
\end{equation}
where $\cA$ is the appropriate torus and $d\ba$ is its Haar measure; this identity establishes part~(a).
Let $\nu$ be the pushforward, under $\Psi$, of Haar measure on $\cA$ to a measure on $\R^r$. By the change of variables formula~\eqref{pushforward change of variables},
\begin{equation*}
     \int_{A} (f \circ \Psi) (\ba) \,d\ba =
    \int_{\R^r} f(\bx) \,d \nu(\bx);
\end{equation*}
in combination with equation~\eqref{ld1}, this identity establishes part~(b).
\end{proof}

The next lemma, which is~\cite[Proposition~5.1]{D}, provides a consequence of the above result in the case that $\{\xi_1, \ldots, \xi_k \}$ can be 
divided into two relatively independent sets (recall Definition~\ref{relind} in Section~\ref{convolution section}), namely the decomposition of the limiting torus into the direct product of two subtori. 

\begin{lemma}  \label{relatively independent lemma Devin}
Let $k_1$ and $k_2$ be positive integers and set $k=k_1+k_2$.
Let $\{ \xi_1, \ldots, \xi_k \}$ be real numbers, and suppose that $\{ \xi_1, \dots, \xi_{k_1} \}$ and $\{\xi_{k_1+1}, \dots, \xi_k\}$ are relatively independent. Define the following subtori which are the closures of certain rays:
\begin{align}
\cA_1 &\text{ is the closure of } \{ (t\xi_1, \dots, t\xi_{k_1}) \colon t\in\R \} \text{ in } \T^{k_1}; \notag \\
\cA_2 &\text{ is the closure of } \{ (t\xi_{k_1+1}, \dots, t\xi_k) \colon t\in\R \} \text{ in } \T^{k_2}; \label{three closures} \\
\cA &\text{ is the closure of } \{ (t\xi_1, \dots, t\xi_{k}) \colon t\in\R \} \text{ in } \T^{k}. \notag
\end{align}
Then $\cA = \cA_1 \times \cA_2$,
and in particular, the normalized Haar measure~$d \ba$ on~$\cA$ equals the product $d \ba = d \ba_1 \,d \ba_2$ of the normalized Haar measures on the subtori~$\cA_1$ and~$\cA_2$.
\end{lemma}

\noindent A key point of this lemma is that the relative independence of the two sets is necessary; the lemma would be false if, for example, $k_1=k_2=1$ and $\xi_1=\xi_2$.

Using the results stated so far in this appendix, we can establish a convolution relationship among the limiting distributions of various functions defined by trigonometric polynomials.

\begin{lemma}  \label{kwconvolution}
Let $k_1$ and $k_2$ be positive integers and set $k=k_1+k_2$.
Let $\{ \xi_1, \ldots, \xi_k \}$ be real numbers, and suppose that $\{ \xi_1, \dots, \xi_{k_1} \}$ and $\{\xi_{k_1+1}, \dots, \xi_k\}$ are relatively independent.
Let $\{ \bz_1, \ldots, \bz_k \} \in \C^r$, and define functions $\Psi_1\colon \T^{k_1} \to \R^r$ and $\Psi_2\colon \T^{k_2} \to \R^r$ and $\Psi\colon \T^{k} \to \R^r$ by
\begin{align*}
\Psi_1(\zeta_1, \ldots, \zeta_{k_1}) &= 2 \Re \bigg(  \sum_{j=1}^{k_1} \bz_j e^{2 \pi i \zeta_j} \bigg), \\
\Psi_2(\zeta_1, \ldots, \zeta_{k_2}) &= 2 \Re \bigg(  \sum_{j=1}^{k_2} \bz_{k_1+j} e^{2 \pi i \zeta_j} \bigg), \\
\Psi(\zeta_1, \ldots, \zeta_k) &= 2 \Re \bigg(  \sum_{j=1}^k \bz_j e^{2 \pi i \zeta_j} \bigg).
\end{align*}
Define the corresponding functions from $\R$ to $\R^r$,
\begin{align*}
  \eta_1(t) &= \Psi_1(t \xi_1, \ldots, t\xi_{k_1}), \\
  \eta_2(t) &= \Psi_2(t \xi_{k_1+1}, \ldots, t\xi_k), \\
  \eta(t) &= \Psi(t \xi_1, \ldots, t\xi_k),
\end{align*}
and let $\nu_1$, $\nu_2$, and $\nu$ be their limiting distributions (as guaranteed by Corollary~\ref{kwcorollary}). Then $\nu = \nu_1 * \nu_2$.
\end{lemma}

\begin{proof}
Let $f(\bx)$ be a bounded, continuous function on $\R^r$;
we need to show that $\int_{\R^r} f(\bx) \,d\nu(\bx) = \int_{\R^r} f(\bx) \,d(\nu_1*\nu_2)(\bx)$.
Let $\cA_1$, $\cA_2$, and $\cA$ be the subtori defined as the closures of certain rays:
\begin{align}
\cA_1 &\text{ is the closure of } \{ (t\xi_1, \dots, t\xi_{k_1}) \colon t\in\R \} \text{ in } \T^{k_1}; \notag \\
\cA_2 &\text{ is the closure of } \{ (t\xi_{k_1+1}, \dots, t\xi_k) \colon t\in\R \} \text{ in } \T^{k_2}; \label{three closures} \\
\cA &\text{ is the closure of } \{ (t\xi_1, \dots, t\xi_{k}) \colon t\in\R \} \text{ in } \T^{k}. \notag
\end{align}
We let $d \ba_1$, $d \ba_2$, and  $d \ba$  denote Haar measure on $\cA_1$, $\cA_2$, and $\cA$, respectively.
By Corollary \ref{kwcorollary} we have 
\begin{equation}
   \label{changeofvarA}
   \int_{\mathbb{R}^r} f(x) \, d\nu(\bx) =  \int_{\cA} (f \circ \Psi)(\ba) \,d\ba
\end{equation}
and, for $j=1$ and $j=2$,
\begin{equation}
  \label{changeofvarAj}
    \int_{\mathbb{R}^r} f(x) \, d\nu_j(\bx) =  \int_{\cA_j} (f \circ \Psi_j)(\ba_j) \,d\ba_j.
\end{equation}
On the other hand, by the definition of convolution and the Fubini--Tonelli theorem,
\begin{align*}
  \int_{\R^r}  f(x) \,d (\nu_1*\nu_2)(\bx) &= \int_{\R^r}
  \Big( 
   \int_{\R^r} f(\bx_1+\bx_2)  \,d\nu_1(\bx_1) \, \Big) d\nu_2(\bx_2)  
   = \int_{\R^r} g(\bx_2) \,d\nu_2(\bx_2) \\
   & = \int_{\cA_2} (g \circ \Psi_2)(\ba_2) \, d \ba_2
\end{align*}
by equation~\eqref{changeofvarAj}, since 
 $g(\bx_2) = \int_{\R^r} f(\bx_1 +\bx_2) \,d \nu_1(\bx_1)$ is bounded and continuous by~\cite[Theorem~16.8]{Bi}.
Now observe that equation~\eqref{changeofvarAj} implies that
\begin{equation}
   (g \circ \Psi_2)(\ba_2)
   = \int_{\R^r} f(x_1 +  \Psi_2(\ba_2) ) \, d \nu_1(x) 
   = \int_{\cA_1}  f(\Psi_1(\ba_1) +  \Psi_2(\ba_2) ) \, d \ba_1,
\end{equation}
and thus by Fubini's theorem,
\begin{align*}
  \int_{\R^r}  f(x) \,d (\nu_1*\nu_2)(\bx)  &=
    \int_{\cA_2}\Big(  \int_{\cA_1}  f(\Psi_1(\ba_1) +  \Psi_2(\ba_2) ) \, d \ba_1  \Big)  \, d \ba_2 \\
    &= \int_{\cA_1 \times \cA_2} 
     f(\Psi_1(\ba_1) +  \Psi_2(\ba_2) ) d (\ba_1 \times \ba_2). 
\end{align*}
By Lemma~\ref{relatively independent lemma Devin} we have $\cA= \cA_1 \times \cA_2$,
and thus this integral equals the integral in equation~\eqref{changeofvarA}. 
\end{proof}

\section*{Acknowledgments}

Research for this article was  supported by NSERC Discovery grants.
The authors thank the Banff International Research Station for their hospitality in May 2009 for the Research in Teams meeting
{\it Prime number races and zeros of Dirichlet L-functions}; we greatly appreciated the excellent working conditions and the beautiful scenery of the Canadian Rockies. 
We also thank Lucile Devin, Peter Humphries, Dave Morris, and Lior Silberman for discussions relating to the proof of
the Kronecker--Weyl theorem 
and Ed Perkins for discussions related to probability.

\end{document}